\documentclass[12pt, reqno]{amsart}
\usepackage{amssymb}
\usepackage{eucal}
\usepackage{amsmath}
\usepackage{amscd}
\usepackage{txfonts}      
\usepackage[dvips]{color}
\usepackage{mathtools}
\usepackage{multicol}
\usepackage[all]{xy}           
\usepackage{graphicx}
\usepackage{color}
\usepackage{colordvi}
\usepackage{xspace}
\usepackage{ulem}
\usepackage{cancel}
\usepackage{tikz}
\usepackage{longtable}
\usepackage{multirow}
\usepackage{ifpdf}
\ifpdf
 \usepackage[colorlinks,final,backref=page,hyperindex]{hyperref}
\else
 \usepackage[colorlinks,final,backref=page,hyperindex,hypertex]{hyperref}
\fi

\makeatletter
\@namedef{subjclassname@2020}{\textup{2020} Mathematics Subject Classification}
\makeatother

\begin{document}
\newcommand {\emptycomment}[1]{} 

\newcommand{\tabincell}[2]{\begin{tabular}{@{}#1@{}}#2\end{tabular}}

\newcommand{\nc}{\newcommand}
\newcommand{\delete}[1]{}

\nc{\mlabel}[1]{\label{#1}}  
\nc{\mcite}[1]{\cite{#1}}  
\nc{\mref}[1]{\ref{#1}}  
\nc{\meqref}[1]{~\eqref{#1}} 
\nc{\mbibitem}[1]{\bibitem{#1}} 

\delete{
\nc{\mlabel}[1]{\label{#1} {{\tiny\tt (#1)}}\ }  
\nc{\mcite}[1]{\cite{#1} {{\tiny\tt (#1)}}\ }  
\nc{\mref}[1]{\ref{#1}{{\tiny\tt (#1)}}\ }  
\nc{\meqref}[1]{~\eqref{#1}{{\tiny\tt (#1)}}\ } 
\nc{\mbibitem}[1]{\bibitem[\bf #1]{#1}} 
}

\newtheorem{thm}{Theorem}[section]
\newtheorem{lem}[thm]{Lemma}
\newtheorem{cor}[thm]{Corollary}
\newtheorem{pro}[thm]{Proposition}
\newtheorem{conj}[thm]{Conjecture}
\theoremstyle{definition}
\newtheorem{defi}[thm]{Definition}
\newtheorem{ex}[thm]{Example}
\newtheorem{rmk}[thm]{Remark}
\newtheorem{pdef}[thm]{Proposition-Definition}
\newtheorem{condition}[thm]{Condition}

\renewcommand{\labelenumi}{{\rm(\alph{enumi})}}
\renewcommand{\theenumi}{\alph{enumi}}
\renewcommand{\labelenumii}{{\rm(\roman{enumii})}}
\renewcommand{\theenumii}{\roman{enumii}}

\nc{\tred}[1]{\textcolor{red}{#1}}
\nc{\tblue}[1]{\textcolor{blue}{#1}}
\nc{\tgreen}[1]{\textcolor{green}{#1}}
\nc{\tpurple}[1]{\textcolor{purple}{#1}}
\nc{\btred}[1]{\textcolor{red}{\bf #1}}
\nc{\btblue}[1]{\textcolor{blue}{\bf #1}}
\nc{\btgreen}[1]{\textcolor{green}{\bf #1}}
\nc{\btpurple}[1]{\textcolor{purple}{\bf #1}}

\nc{\vsa}{\vspace{-.1cm}}
\nc{\vsb}{\vspace{-.2cm}}
\nc{\vsc}{\vspace{-.3cm}}
\nc{\vsd}{\vspace{-.4cm}}
\nc{\vse}{\vspace{-.5cm}}


\newcommand{\End}{\text{End}}

\nc{\calb}{\mathcal{B}}
\nc{\call}{\mathcal{L}}
\nc{\calo}{\mathcal{O}}
\nc{\calp}{\mathcal{P}}
\nc{\frakg}{\mathfrak{g}}
\nc{\frakh}{\mathfrak{h}}
\nc{\ad}{\mathrm{ad}}

\nc{\ccred}[1]{\tred{\textcircled{#1}}}

\nc{\ddop}{dual a-$\mathcal{O}$-operator\xspace}

\nc{\dop}{a-$\mathcal{O}$-operator\xspace} \nc{\sdpp}{special
apre-perm algebra\xspace} \nc{\sdpps}{special apre-perm
algebras\xspace} \nc{\sdppb}{special apre-perm bialgebra\xspace}

\nc{\sdppbs}{special apre-perm bialgebras\xspace}
\nc{\dpp}{apre-perm algebra\xspace} \nc{\dpps}{apre-perm
algebras\xspace} \nc{\da}{APP algebra\xspace} \nc{\das}{APP
algebra\xspace} \nc{\sa}{SAPP algebra\xspace} \nc{\sas}{SAPP
algebras\xspace}

\nc{\sdppsubs}{special apre-perm subalgebras\xspace}


\newcommand{\cm}[1]{\textcolor{red}{\underline{CM:}#1 }}
\newcommand{\gl}[1]{\textcolor{blue}{\underline{GL:}#1 }}
\nc{\li}[1]{\textcolor{purple}{#1}}
\nc{\lir}[1]{\textcolor{purple}{Li: #1}}

\title[Induced structures of averaging infinitesimal bialgebras]{Induced structures of averaging commutative and cocommutative infinitesimal bialgebras via a new splitting of perm
algebras}

    \author{Chengming Bai}
    \address{Chern Institute of Mathematics \& LPMC, Nankai University, Tianjin 300071, China}
    \email{baicm@nankai.edu.cn}

    \author{Li Guo}
    \address{Department of Mathematics and Computer Science, Rutgers University, Newark, NJ 07102, USA}
    \email{liguo@rutgers.edu}

    \author{Guilai Liu}
    \address{Chern Institute of Mathematics \& LPMC, Nankai University, Tianjin 300071, China}
    \email{liugl@mail.nankai.edu.cn}

    \author{Quan Zhao}
    \address{Chern Institute of Mathematics \& LPMC, Nankai University, Tianjin 300071, China}
    \email{zhaoquan@mail.nankai.edu.cn}

\date{\today}%

\begin{abstract}
It is well-known that an averaging operator on a commutative associative algebra gives rise to a perm algebra. This paper lifts this process to the level of bialgebras. For this purpose, we first give an infinitesimal bialgebra structure for averaging commutative associative algebras and characterize it by double constructions of averaging Frobenius commutative algebras. To find the bialgebra counterpart of perm algebras that is induced by such averaging bialgebras,
we need a new two-part splitting of the multiplication in a perm algebra, which differs from the usual splitting of the perm algebra (into the pre-perm algebra) by the characterized representation.
This gives rise to the notion of an averaging-pre-perm algebra, or simply an apre-perm algebra.
Furthermore, the notion of special apre-perm algebras which are apre-perm algebras with the second
multiplications being commutative is introduced as the underlying
algebra structure of perm algebras with
nondegenerate symmetric left-invariant bilinear forms. The latter are
also the induced structures of symmetric Frobenius commutative
algebras with averaging operators. Consequently, a double
construction of averaging Frobenius commutative algebra gives rise
to a Manin triple of special apre-perm algebras. In terms of bialgebra structures, this means that an
averaging commutative and cocommutative infinitesimal bialgebra
gives rise to a special apre-perm bialgebra.
\end{abstract}

\subjclass[2020]{
    17A36,  
    17A40,  
    17B10, 
    17D25,  
    18M70.  
}

\keywords{Averaging operator, perm algebra, bialgebra, $\calo$-operator, embedding tensor}

\maketitle

\vspace{-1.3cm}

\tableofcontents

\vspace{-.5cm}

\allowdisplaybreaks

\section{Introduction}
We enhance the well-known connection that an averaging commutative associative algebra induces a perm algebra to the context of bialgebras, with the induced bialgebra given by a new splitting of perm algebras.
\vsb

\subsection{Averaging operators and perm algebras}\

Let $\mathcal{P}$ be a binary operad with one binary operation and $(A,\cdot_{A})$ be a
$\mathcal{P}$-algebra. An {\bf averaging operator} on
$(A,\cdot_{A})$ is a linear map $P:A\rightarrow A$ such that
\vsb
\begin{equation}\mlabel{eq:Ao}
    P(x)\cdot_{A}P(y)=P\big(P(x)\cdot_{A}y\big)=P\big(x\cdot_{A} P(y)\big),\;\forall x,y\in A.
\vsa
\end{equation}
Then we call $(A,\cdot_{A},P)$ an {\bf averaging
$\mathcal{P}$-algebra}. Averaging operators for associative algebras originated from the 1895 work~\mcite{Re} of O. Reynolds on turbulence theory in fluid mechanics and those of Kolmogoroff and of Kamp\'e de
F\'eriet in probability~\mcite{KF}. In the last century, the operators were studied extensively by many authors, including G. Birkhoff, Miller  and Rota \cite{Bi,Mil,R2}.
On Lie algebras, an averaging operator arose from an
embedding tensor whose study can be traced back to gauged supergravity theory  \mcite{NS}  and attracts more
attention in mathematical physics  \mcite{KS,L,STZ,TS}. Recently,
averaging operators find applications in combinatorics, number
theory, operads, cohomology and deformation
theory \mcite{DS,GuK,GZ,PBGN2,PG,WZ,ZG}.

This paper focuses on averaging operators on commutative associative
algebras, that is, the {\bf averaging commutative associative
algebras}. In this case, \eqref{eq:Ao} simplifies to
\vsa
\begin{equation}\mlabel{eq:aver op}
P(x)\cdot_{A}P(y)=P\big(P(x)\cdot_{A} y\big),\;\;\forall x,y\in A.
\vsa
\end{equation}

On the other hand, a {\bf perm algebra} \mcite{Chap2001} is a vector
space $A$ together with a binary operation  $\circ _{A}:A\otimes
A\rightarrow A$ satisfying
\vsb
\begin{equation}\mlabel{eq:perm}
x\circ _{A}(y\circ _{A}z)=(x\circ _{A}y)\circ _{A}z=(y\circ
_{A}x)\circ _{A}z,\;\;\forall x,y,z\in A.
\vsb
\end{equation}
Perm algebras play an important role in algebraic operad theory,
since their operad is the Koszul dual to the operad of pre-Lie
algebras \mcite{LV}, as well as the duplicator of the operad of
commutative associative algebras \mcite{GuK,PBGN}.

These two structures are connected by the following result of Aguiar, in the spirit of S. Gelfand's theorem that a differential commutative associative algebra gives rise to a Novikov algebra~\mcite{GD}.
\vsb
\begin{pro}\mlabel{ex:comm aver}\mcite{Agu2000*}
Let $P$ be an averaging operator on a commutative associative algebra
$(A,\cdot_{A})$. Then there is a perm algebra $(A,\circ_{A})$
given by
\vsb
\begin{equation}\mlabel{eq:perm from aver op}
x\circ_{A} y=P(x)\cdot_{A} y, \;\;\forall x,y\in A.
\vsb
\end{equation}
\end{pro}

As it turns out, this is just one instance of the general phenomenon that an averaging operator on a
$\mathcal P$-algebra gives rise to a ${\rm Du}({\mathcal P})$-algebra,
where ${\rm Du}({\mathcal P})$ is the duplicator of the operad
$\mathcal P$ \mcite{GuK,PBGN,PBGN2} (called di-Var-algebras in~\mcite{GuK}).

In addition to differential algebras and averaging algebras, applications in mathematics and physics have given rise to other algebras that are equipped with linear operators, such as Rota-Baxter operators, Nijenhuis operators and Reynolds operators. These structures are unified under the term of {\bf $\calp$-operated algebras}, defined simply as a $\calp$-algebra equipped with a linear operator that usually satisfies an operator identity \mcite{Gop}.
\vsb
\subsection{Bialgebras of operated algebras and their induced bialgebras}
\vsa
\subsubsection{Bialgebras from the Manin triple approach}
A bialgebra structure is a vector space equipped with an algebra
structure and a coalgebra structure of the same type tied together by compatible conditions. Well-known bialgebra structures include Lie
bialgebras \mcite{Cha,Dri} which are closely related to Poisson-Lie
groups and play an important role in the infinitesimalization of
quantum groups, and antisymmetric infinitesimal (ASI) bialgebras
\mcite{Agu2000, Agu2001, Agu2004, Bai2010} which can be used to
construct symmetric Frobenius algebras and thus find applications
in 2d topological and string theory \mcite{Kock,Lau}.

All these bialgebras have a common property that they are
equivalently characterized by Manin triples associated to
nondegenerate bilinear forms on the corresponding algebras
satisfying certain conditions. Explicitly, a Lie bialgebra is
equivalent to a Manin triple of Lie algebras (associated to the
nondegenerate symmetric invariant bilinear form) and an ASI
bialgebra is equivalent to a double construction of Frobenius
algebra (that is, a Manin triple of associative algebras associated to
the nondegenerate symmetric invariant bilinear form). Such an
approach has been successively applied to many other types of
algebras such as pre-Lie algebras \mcite{Bai2008} and Poisson
algebras \mcite{NB}.
\vsb
\subsubsection{Bialgebras of operated algebras}

Given the broad applications of various operated algebras, it is
important to study the corresponding bialgebras, that is, to
extend the above bialgebra theory of algebra structures to the
context of operated algebras. Some of these bialgebra theories
have been established recently, including Rota-Baxter associative
algebras and Lie algebras \mcite{BGLM,BGM}, differential
associative algebras \mcite{LLB}, and endo Lie algebras
\mcite{BGS}.

These operated bialgebras also have equivalent interpretations in terms of Manin triples of the corresponding algebras with the linear operators satisfying
certain compatibility conditions. However, integrating the linear operators with the compatibility conditions of the original bialgebras introduces additional complexity, reflected as the admissible conditions
between the linear operators on the algebras and the coalgebras.
\vsb
\subsubsection{Induced bialgebras from operated bialgebras}
Various operated algebras are known to induce new structures. Other than the aforementioned instances of
differential commutative associative algebras inducing
Novikov algebras and averaging $\mathcal P$-algebras inducing
${\rm Du}(\mathcal P)$-algebras, Rota-Baxter $\mathcal P$-algebras
induce ${\rm Su}(\mathcal P)$-algebras, where ${\rm Su}(\mathcal
P)$ is the successor of the operad $\mathcal P$ \mcite{BBGN}.
Typical examples are dendriform algebras from Rota-Baxter
associative algebras and pre-Lie algebras from Rota-Baxter Lie
algebras.

Moving to the bialgebra level, it is also interesting to explore the induced bialgebra structures of the various operated bialgebras. So far, this process remains quite mysterious, since the patterns of induced structures on the algebra level no longer apply.
For example, for bialgebras for Rota-Baxter Lie algebras,
the induced structures are special L-dendriform bialgebras
\mcite{BGLM} that relates to Lie groups with left-invariant flat
pseudo-metrics in geometry \mcite{BHC}, instead of the bialgebras
for pre-Lie algebras \mcite{Bai2008}, whereas in the case of
bialgebras for Rota-Baxter (associative) algebras \mcite{BGM}, the induced
structures are a subclass of quadri-bialgebras \mcite{NB2}, instead
of the various bialgebra theories for dendriform algebras
\mcite{Bai2010}. Moreover, only under certain necessary conditions
commutative and cocomutative differential ASI bialgebras give
Novikov bialgebras \mcite{HBG} and Poisson bialgebras \mcite{NB}.
\vsb
\subsection{A new splitting of perm algebras and \sdppbs}\

This paper aims to develop a bialgebra structure for  averaging commutative associative algebras and to lift the induction in Proposition \mref{ex:comm aver} to the level of bialgebras.

\subsubsection{Challenges in constructing the induced
bialgebras from averaging commutative and cocommutative
infinitesimal bialgebras} As the first step, we develop a
bialgebra structure for averaging commutative associative algebras,
namely averaging commutative and cocommutative
infinitesimal bialgebras. The
main idea follows the aforementioned approach given in
\mcite{BGM}, that is,  averaging commutative and cocommutative
infinitesimal bialgebras are equivalently characterized by the introduced
notion of double constructions of averaging Frobenius commutative
algebras.

On the other hand, there is a bialgebra theory for perm algebras
(\mcite{LZB}, see also \mcite{Hou}) that are characterized in terms of Manin triples of perm algebras
associated to nondegenerate invariant bilinear forms. Here a bilinear form $\mathcal B$ on a perm algebra
$(A,\circ_A)$ is called {\bf invariant} if $\mathcal B$ is antisymmetric and satisfies
\vsa
\begin{equation}
\mathcal B(x\circ_A y, z)=\mathcal B(x, y\circ_A z-z\circ_A
y),\;\;\forall x,y,z\in A.
\end{equation}

We observe that this perm bialgebra does not serve the purpose of
the ``induced structure" of an averaging commutative and
cocommutative infinitesimal bialgebra. Indeed, the former
corresponds to a Manin triple of perm algebras whose associated
bilinear form is {\it antisymmetric}, whereas the latter should be induced from a
double construction of averaging Frobenius commutative algebras
whose associated bilinear form is {\it symmetric}. Thus to lift
Proposition \mref{ex:comm aver} to the level of bialgebras
requires a new kind of bilinear forms.

\subsubsection{ Induced structures from symmetric Frobenius commutative algebras with averaging operators}
A critical observation in finding such a new bilinear form is that, when an averaging commutative associative algebra is equipped with a nondegenerate symmetric invariant bilinear form (that is, a symmetric Frobenius commutative algebra $(A,\cdot_A,\mathcal B)$ with an
averaging operator $P$), the induced perm algebra $(A,\circ_A)$ from Proposition \mref{ex:comm aver} is also equipped with the bilinear form $\mathcal B$ which is {\bf left-invariant} in the sense that
\vsa
\begin{equation*}
\mathcal{B}(x\circ_{A}y,z)=\mathcal{B}(y,x\circ_{A}z),\;\;\forall
x,y,z\in A.
\vsa
\end{equation*}
This observation leads us to choosing the Manin triples of perm
algebras associated to nondegenerate symmetric left-invariant
bilinear forms as a candidate for the desired induced structures of double constructions of averaging Frobenius commutative algebras, or equivalently, the
bialgebra structures corresponding to the former as the induced
structures of averaging commutative and cocommutative infinitesimal
bialgebras.

In order to make this idea work, we first need to find the
underlying algebra structures of perm algebras with nondegenerate
symmetric left-invariant bilinear forms, in the sense
that these structures are obtained from the latter such that they
give representations of the perm algebras on the dual spaces which
are equivalent to the adjoint representations of the perm
algebras.  Motivated by the
previous studies on the induced structures
of bialgebras for Rota-Baxter algebras \mcite{BGM} and
Rota-Baxter Lie algebras \mcite{BGLM}, such underlying structures
should be related to ``splitting" the operation of the perm
algebra in a suitable sense.

\subsubsection{A new splitting of perm algebras}
Let us recall the usual splitting operations in the sense of
successors of operads \mcite{BBGN}. For a binary operad $\mathcal
P$ with one binary operation, the notion of a {\bf pre-$\mathcal
P$ algebra} $(A,\triangleright_{A},\triangleleft_{A})$ is a vector
space $A$ with multiplications
$\triangleright_{A},\triangleleft_{A}:A\otimes A\rightarrow A$,
such that $(A,\circ_{A})$ given by
\vsa
\begin{equation}\mlabel{eq:Dpp}
x\circ_{A}y=x\triangleright_{A} y+x\triangleleft_{A} y,\;\forall
x,y\in A
\vsa
\end{equation}
is a $\mathcal{P}$-algebra and
$(\mathcal{L}_{\triangleright_{A}},\mathcal{R}_{\triangleleft_{A}},A)$
is a representation of $(A,\circ_{A})$, where
$\mathcal{L}_{\triangleright_{A}}(x)$ and
$\mathcal{R}_{\triangleleft_{A}}(x)$ for all $x\in A$ are left and
right multiplication operators of $\triangleright_{A}$ and
$\triangleleft_{A}$ respectively. In the case of perm algebras,
the notion of pre-perm algebras has been given in \mcite{LZB}.

However, this splitting does not serve our purpose, since pre-perm
algebras are the underlying  algebra structures of
perm algebras with nondegenerate symmetric bilinear forms which do not satisfy our required left-invariant condition. Consequently, we have to find a new splitting of perm algebras as follows.

\begin{defi}\mlabel{defi:1.2}
An {\bf apre-perm algebra} (with apre in short for averaging-pre)
$(A,\triangleright_{A},\triangleleft_{A})$ is a vector space $A$
with multiplications
$\triangleright_{A},\triangleleft_{A}:A\otimes A\rightarrow A$,
such that $(A,\circ_{A})$ defined by \meqref{eq:Dpp} is a {perm
algebra} and
$(\mathcal{L}^*_{\triangleright_{A}}+\mathcal{R}^*_{\triangleleft_{A}},-\mathcal{R}^*_{\triangleleft_{A}},A^*)$
is a representation of $(A,\circ_{A})$ where, for all $x\in A$,
$\mathcal{L}^*_{\triangleright_{A}}(x)$ and
$\mathcal{R}^*_{\triangleleft_{A}}(x)$ are the transpose maps of
$\mathcal{L}_{\triangleright_{A}}(x)$ and
$\mathcal{R}_{\triangleleft_{A}}(x)$ respectively.
\vsb
\end{defi}

Pre-perm algebras and apre-perm algebras are defined by different
representations of perm algebras and hence are
different splitting of the operation of perm algebras.
In particular, an \dpp $(A,\triangleright_A,\triangleleft_A)$ is
called {\bf special} if $\triangleleft_{A}$ is commutative. We
show that \sdpps are the underlying algebra structures of perm
algebras with nondegenerate symmetric left-invariant bilinear
forms. 
More explicitly,  a perm algebra
$(A,\circ_A)$ with a nondegenerate symmetric left-invariant
bilinear form induces a \sdpp
$(A,\triangleright_A,\triangleleft_A)$. And there is a one-to-one
correspondence between the former and \sdpps with nondegenerate
symmetric bilinear forms satisfying certain conditions
corresponding to the left-invariance. Note that the
left-invariance of the nondegenerate bilinear form corresponds to
the equivalence between the adjoint representation of the perm
algebra $(A,\circ_A)$ and
$(\mathcal{L}_{\triangleright_A}^*+\mathcal{R}_{\triangleleft_A}^*,
-\mathcal{R}_{\triangleleft_A}^*, A^*)$ as the representations of
$(A,\circ_A)$, or equivalently, the equivalence between the
adjoint representation and the coadjoint representation of the
\sdpp $(A,\triangleright_A,\triangleleft_A)$. 

Eventually, we introduce the notion of a \sdppb, as the
equivalent structure of a Manin triple of
perm algebras associated to a nondegenerate symmetric
left-invariant bilinear form. Equivalently, the corresponding
Manin triple of \sdpps is obtained naturally from a
double construction of averaging Frobenius commutative algebra.
Consequently the \sdppb can be regarded as the desired induced
bialgebra structure of an averaging commutative and cocommutative
infinitesimal bialgebra.

\subsection{Organization of the paper}\

This paper is organized as follows.

In Section \mref{sec:2}, we study the representations of averaging
commutative associative algebras. The notion of admissible
averaging commutative associative algebras is introduced from the
representations of averaging commutative associative algebras on
the dual spaces. Then we introduce the notions of double
constructions of averaging Frobenius commutative algebras and
averaging commutative and cocommutative infinitesimal bialgebras.
The two structures are shown to be equivalent (Theorem~\ref{thm:2.11}).

In Section \mref{sec:4}, we first introduce the notion of an \dpp as a new splitting of the operation of perm
algebras. Then the notion of a \sdpp is introduced as the
underlying algebra structure of a nondegenerate symmetric
left-invariant bilinear form on a perm algebra
(Corollary~\ref{cor:quadratic1} and Proposition~\ref{pro:330}). We
also prove that an admissible averaging commutative associative
algebra induces a \sdpp (Proposition~\ref{pro:com asso and SDPP}),
and a \sdpp gives a pre-Lie algebra as well as an anti-pre-Lie algebra.

In Section \mref{sec:5}, we introduce the notions of a Manin triple of perm algebras associated to a nondegenerate
symmetric left-invariant bilinear form, a Manin triple
of \sdpps, and \sdppbs. The equivalences of these notions are
established. It follows that a double construction of averaging
Frobenius commutative algebra induces a Manin triple of \sdpps
(Proposition~\ref{pro:5.2}). Equivalently from the bialgebra
viewpoint, a averaging commutative and cocommutative infinitesimal
bialgebra induces a \sdppb (Proposition \ref{pro:5.9}).

\noindent
{\bf Notations. }
Throughout this paper, unless otherwise specified, all the vector
spaces and algebras are finite-dimensional over an algebraically
closed field $\mathbb {K}$ of characteristic zero, although many
results and notions remain valid in the infinite-dimensional case, and by an algebra, we mean an associative algebra. For a vector space $A$, let
\vsb
$$\tau:A\otimes A\rightarrow A\otimes A,\quad x\otimes y\mapsto y\otimes x,\;\;\;\forall x,y\in A,
\vsa
$$
be the flip operator. Let  $\circ_A:A\otimes A\rightarrow A$ be a binary operation on $A$. Define linear maps ${\mathcal L}_{\circ_A},
{\mathcal R}_{\circ_A}:A\rightarrow {\rm End}_{\mathbb K}(A)$
respectively by
\vsb
\begin{eqnarray*}
    {\mathcal L}_{\circ_A}(x)y:=x\circ_A y,\;\; {\mathcal
        R}_{\circ_A}(x)y:=y\circ_A x, \;\;\;\forall x, y\in A.
\vse
\end{eqnarray*}

\section{Averaging commutative and cocommutative infinitesimal bialgebras}\mlabel{sec:2}\

This section studies representations of averaging commutative
algebras. Such a representation on the dual space
leads to the notion of an admissible averaging commutative
algebra. In particular, there exists a natural
admissible averaging commutative algebra structure on
a (symmetric) Frobenius commutative algebra with an averaging
operator. The equivalence between double constructions of Frobenius commutative algebras and commutative and cocommutative
infinitesimal bialgebras is enriched to the case when both structures carry suitably defined averaging operators.

\subsection{Representations of averaging commutative algebras}\
We recall the following notion.

\begin{defi}
Let $(A,\cdot_A)$ be a commutative algebra. If a
linear map $P:A\rightarrow A$ satisfies (\mref{eq:aver op}), then
$P$ is called an {\bf averaging operator} on $(A,\cdot_A)$, and $(A,\cdot_A,P)$ is called an {\bf averaging commutative algebra}.
\vsb
\end{defi}

We introduce the notion of representations of averaging
commutative algebras.

\begin{defi}
Let $(A,\cdot_{A},P)$ be an averaging commutative
algebra. Let $(\mu,V)$ be a representation of
$(A,\cdot_{A})$, given by a linear map $\mu:A\rightarrow\mathrm{End}_{\mathbb
K}(V)$ satisfying
\vsa
\begin{equation*}
\mu (x\cdot_{A}y)v=\mu (x)\mu (y)v,\;\forall x,y\in A,v\in V.
\vsa
\end{equation*}
Let $\alpha:V\rightarrow V$ be a linear map such that
\vsa
\begin{equation}\mlabel{eq:rep ao}
\mu(Px)\alpha(v)=\alpha\big(\mu(Px)v\big)=\alpha\big(\mu(x)\alpha(v)\big),\;\forall
x\in A, v\in V.
\vsa
\end{equation}
Then the triple $(\mu,\alpha,V)$ is called a {\bf representation}
of $(A,\cdot_{A},P)$. Two representations
$(\mu_{1},\alpha_{1},V_{1})$ and $(\mu_{2},\alpha_{2},V_{2})$ of
$(A,\cdot_{A},P)$ are called {\bf equivalent} if there exists a
linear isomorphism $\phi:V_{1}\rightarrow V_{2}$ such that
\vsb
\begin{equation*}
\phi\big(\mu_{1}(x)v\big)=\mu_{2}(x)\phi(v),\; \phi\alpha_{1}(v)=\alpha_{2}\phi(v),\;\forall x\in A, v\in V_{1}.
\end{equation*}
\end{defi}

\begin{ex}
Let $(A,\cdot_{A},P)$ be an averaging commutative algebra.
Then  $(\mathcal{L}_{\cdot_{A}},P,A)$ is a
representation of $(A,\cdot_{A},P)$, which is called the {\bf
adjoint representation} of $(A,\cdot_{A},P)$.
\end{ex}

For vector spaces $V_1$ and $V_2$ and linear maps
$\phi_1:V_1\rightarrow V_1$ and $\phi_{2}:V_2\rightarrow V_2$, define
\vsa
\begin{equation}\mlabel{eq:pro:SD RB Lie2}
\phi_1+\phi_2:V_1\oplus V_2\rightarrow V_1\oplus
V_2,\;\;v_1+v_2\mapsto \phi_1(v_1)+\phi_2(v_2),\;\;\forall v_1\in
V_1,v_2\in V_2.
\vsb
\end{equation}

A straightforward verification gives
\begin{pro}
Let $(A,\cdot_{A},P)$ be an averaging commutative
algebra. Let $V$ be a vector space, and
$\mu:A\rightarrow\mathrm{End}_{\mathbb K}(V),\;
\alpha:V\rightarrow V$ be linear maps. Then $(\mu,\alpha,V)$ is a
representation of $(A,\cdot_{A},P)$ if and only if there is a
commutative algebra structure on the direct sum
$A\oplus V$ of vector spaces  given by
\vsb
\begin{equation*}
(x+u)\cdot_{d}(y+v):=x\cdot_{A}y+\mu(x)v+\mu(y)u,\;\forall  x,y\in A, u,v\in V,
\vsa
\end{equation*}
such that $P+\alpha$ is an averaging operator on $(A\oplus
V,\cdot_{d})$. In this case, we denote the averaging commutative
algebra $(A\oplus V,\cdot_{d},P+\alpha)$ by
$(A\ltimes_{\mu}V,P+\alpha)$.
\end{pro}

Let $\beta:V\rightarrow V$ be a linear map.
Denote the linear map $\beta^{*}:V^{*}\rightarrow V^{*}$ by
\vsa
\begin{equation*}
\langle\beta^{*}(u^{*}),v\rangle=\langle u^{*},\beta(v)\rangle,\;\forall u^{*}\in V^{*},v\in V.
\vsa
\end{equation*}
Moreover, for a linear map $\mu:A\rightarrow\mathrm{End}_{\mathbb
K}(V)$, denote the linear map
$\mu^{*}:A\rightarrow\mathrm{End}_{\mathbb K}(V^{*})$ by
$\mu^{*}(x)=\big(\mu(x)\big)^{*}$, that is,
\vsa
\begin{equation*}
\langle \mu^{*}(x)u^{*},v\rangle= \langle \big(\mu(x)\big)^{*}u^{*},v\rangle =\langle u^{*},\mu(x)v\rangle,\;\forall x\in
A, u^{*}\in V^{*}, v\in V.
\vsa
\end{equation*}

\begin{pro}\mlabel{pro:dual ao}
Let $(A,\cdot_{A},P)$ be an averaging commutative
algebra. Let $\mu:A\rightarrow\mathrm{End}_{\mathbb K}(V)$ and
$\beta:V\rightarrow V$ be linear maps. Then
$(\mu^{*},\beta^{*},V^{*})$ is a representation of
$(A,\cdot_{A},P)$ if and only if  $(\mu,\beta, V)$ is a
representation of $(A,\cdot_{A},P)$, that is, $(\mu,V)$ is a
representation of $(A,\cdot_{A})$ and $\beta$ satisfies the
following equation.
\vsa
\begin{equation}\mlabel{eq:aver pair rep}
\mu(Px)\beta(v)=\beta\big(\mu(Px)v\big)=\beta\big(\mu(x)\beta(v)\big),\;\forall x\in A, v\in V.
\vsa
\end{equation}
In this case, $( A\ltimes_{\mu^{*}}V^{*}, P+\beta^{*} )$   is an averaging   commutative algebra.
\end{pro}
\begin{proof}
By \mcite{Bai2010}, $(\mu^{*},V^{*})$ is a representation of $(A,\cdot_{A})$ if and only if $(\mu,V)$ is a representation of $(A,\cdot_{A})$.
For all $x\in A, u^{*}\in V^{*}, v\in V$, we have
\vsa
\begin{eqnarray*}
\langle \mu^{*}(Px)\beta^{*}(u^{*}),v\rangle&=&\langle u^{*}, \beta\big(\mu(Px)v\big)\rangle,\\
\langle \beta^{*}\big(\mu^{*}(Px)u^{*}\big),v\rangle&=&\langle u^{*},\mu(Px)\beta(v)\rangle,\\
\langle \beta^{*}\big(\mu^{*}(x)\beta^{*}(u^{*})\big),v\rangle&=&\langle u^{*}, \beta\big(\mu(x)\beta(v)\big)\rangle.
\vsa
\end{eqnarray*}
Hence the equation
\vsa
\begin{equation*}
\mu^{*}(Px)\beta^{*}(u^{*})=\beta^{*}\big(\mu^{*}(Px)u^{*}\big)=\beta^{*}\big(\mu^{*}(x)\beta^{*}(u^{*})\big)
\vsa
\end{equation*}
holds if and only if $\beta$ satisfies \meqref{eq:aver pair rep},
which is exactly \meqref{eq:rep ao} when $\alpha$ is replaced by
$\beta$. Therefore the conclusion follows.
\end{proof}

\begin{defi}
Let $(A,\cdot_{A},P)$ be an averaging commutative
algebra and $Q:A\rightarrow A$ be a linear map. We say that {\bf $Q$ is
admissible to $(A,\cdot_{A},P)$} if
\vsa
\begin{equation}\mlabel{eq:ao pair}
P(x)\cdot_{A} Q(y)=Q\big(P(x)\cdot_{A} y\big)=Q\big(x\cdot_{A}Q(y)\big),\;\forall x,y\in A.
\vsa
\end{equation}
In this case, we also call $(A,\cdot_{A},P,Q)$ an {\bf admissible averaging commutative algebra}.
\end{defi}

\begin{cor}
Let $(A,\cdot_{A},P)$ be an averaging commutative algebra.
Then $(\mathcal{L}^{*}_{\cdot_{A}},Q^{*},A^{*})$ is a representation of $(A,\cdot_{A},P)$ if and only if $(A,\cdot_{A},P,Q)$ is an admissible averaging commutative algebra, that is, \meqref{eq:ao pair} holds.
\end{cor}
\begin{proof}
It follows from Proposition \mref{pro:dual ao} by taking $\mu=\mathcal{L}_{\cdot_{A}},\; \beta=Q,\; V=A$.
\end{proof}

\begin{pro}\mlabel{pro:2.6}
Let $(A,\cdot_{A},\mathcal{B})$ be a Frobenius commutative
algebra. Let $P$ be an averaging operator on
$(A,\cdot_{A})$ and $\widehat{P}$ be the adjoint map of $P$ with
respect to $\mathcal{B}$, that is,
\vsb
\begin{equation*}
\mathcal{B}\big(P(x),y\big)=\mathcal{B}\big(x,\widehat{P}(y)\big),\;\forall x,y\in A.
\vsb
\end{equation*}
Then $(A,\cdot_{A},P,\widehat{P})$ is an admissible averaging commutative algebra.
Moreover, $(\mathcal{L}^{*}_{\cdot_{A}},\widehat{P}^{*},$
$A^{*})$ and $(\mathcal{L}_{\cdot_{A}},P,A)$ are equivalent as representations of $(A,\cdot_{A},P)$.
Conversely, let $(A,\cdot_{A},P,Q)$ be an admissible averaging commutative algebra.
If the resulting representation $(\mathcal{L}^{*}_{\cdot_{A}},Q^{*},A^{*})$ of $(A,\cdot_{A},P)$ is equivalent to $(\mathcal{L}_{\cdot_{A}},P,A)$, then there exists a nondegenerate invariant bilinear form $\mathcal{B}$ on $(A,\cdot_{A})$ such that $Q=\widehat{P}$.
\end{pro}
\begin{proof}
The proof follows the same argument as the one of \cite[Proposition 3.9]{BGM}.
By the commutativity of the multiplication $\cdot_{A}$, the bilinear form $\mathcal{B}$ no longer needs to be symmetric herein.
\end{proof}

Note that if a Frobenius algebra $(A,\cdot_{A},\mathcal{B})$ is antisymmetric in the sense that $\mathcal{B}$ is antisymmetric, then the multiplication $\cdot_{A}$ is zero on $A$. In the following examples, which will be revisited several times later, we will only consider symmetric Frobenius algebras.

\begin{ex}\mlabel{pre-ex:2.2}
Let $(A,\cdot_{A})$ be a commutative algebra with a representation $(\mu,V)$. Then the projection $P:A\oplus V\rightarrow A\oplus V$ given by
\vsb
\begin{equation*}\mlabel{eq:aveo}
P(x+u)=x,\;\forall x\in A, u\in V
\vsa
\end{equation*}
is an averaging operator on $A\ltimes _{\mu}V$.
In particular when $V=A^{*}$ and $\mu=\mathcal{L}^{*}_{\cdot_{A}}$, the natural nondegenerate symmetric bilinear form $\mathcal{B}_{d}$ given by
\vsa
\begin{equation}\mlabel{eq:bfds}
\mathcal{B}_{d}(x+a^{*},y+b^{*})=\langle x,b^{*}\rangle+\langle a^{*},y\rangle,\;\forall x,y\in A, a^{*},b^{*}\in A^{*}
\vsa
\end{equation}
is invariant on $A\ltimes_{\mathcal{L}^{*}_{\cdot_{A}}}A^{*}$.
In this case there is an admissible averaging commutative algebra
$(A\ltimes_{\mathcal{L}^{*}_{\cdot_{A}}}A^{*},P,\widehat{P})$, where
\vsb
\begin{equation*}
\widehat{P}(x+a^{*})=a^{*},\;\;\forall  x\in A, a^*\in A^*.
\vsb\end{equation*}
\end{ex}

\begin{ex}\mlabel{pre-ex:2.3}
Let $(A,\cdot_{A})$ be a commutative algebra. Then there is a commutative algebra structure on $A\otimes A$ given by
\vsb
\begin{equation*}
(x\otimes y)\cdot(z\otimes w)=x\cdot_{A} z\otimes y\cdot_{A} w,\;\forall x,y,z,w\in A.
\vsb
\end{equation*}
Moreover, the linear map $P:A\otimes A\rightarrow A\otimes A$ given by
\vsb
\begin{equation*}
P(x\otimes y)=x\otimes y+y\otimes x, \quad \forall x,y\in A,
\vsa
\end{equation*}
is an averaging operator on $(A\otimes A,\cdot)$.
Furthermore, if there is a nondegenerate symmetric invariant bilinear form $\mathcal{B}$ on $(A,\cdot_{A})$, then the nondegenerate symmetric bilinear form $\mathcal{B}'$ on $A\otimes A$ given by
\vsa
\begin{equation*}
\mathcal{B}'(x\otimes y,z\otimes
w)=\mathcal{B}(x,z)\mathcal{B}(y,w), \quad \forall x,y,z,w\in A,
\vsa
\end{equation*}
is invariant. In this case there is an admissible averaging commutative algebra $(A\otimes A,\cdot,P,\widehat{P})$, where
\vsb
\begin{equation*}
\widehat{P}(x\otimes y)=x\otimes y+y\otimes x, \quad \forall
x,y\in A,
\vsa
\end{equation*}
coincides with $P$.
\end{ex}

\begin{ex}\mlabel{ex:multiplication}
Let $(A,\cdot_{A})$ be the 4-dimensional commutative
algebra with a basis $\{ e_{1},e_{2}$, $e_{3},e_{4}\}$ whose
nonzero products are
\vsa
\begin{eqnarray}
e_{1}\cdot_{A}e_{1}=e_{1},\; e_{1}\cdot_{A}e_{2}=e_{2},\; e_{1}\cdot_{A}e_{3}=e_{2}\cdot_{A}e_{4}=e_{3},\; e_{1}\cdot_{A}e_{4}=e_{4}.
\vsb
\end{eqnarray}
Define a linear map $P:A\rightarrow A$ whose nonzero values are given by
\vsb
\begin{equation}\mlabel{eq:ex4}
P(e_{1})=e_{1},\;P(e_{2})=e_{2}.
\vsa
\end{equation}
Then $P$ is an averaging operator on $(A,\cdot_{A})$.
Moreover, there is a nondegenerate symmetric invariant bilinear form $\mathcal{B}$ on $(A,\cdot_{A})$ whose nonzero values are
\vsa
\begin{equation}\mlabel{eq:10.0}
\mathcal{B}(e_{1},e_{3})=\mathcal{B}(e_{2},e_{4})=1.
\vsa
\end{equation}
The adjoint map $\widehat{P}$ of $P$ with respect to $\mathcal{B}$ has its nonzero values given by
\vsa
\begin{equation}
\widehat{P}(e_{3})=e_{3},\;\widehat{P}(e_{4})=e_{4}.
\vsa
\end{equation}
Then by Proposition \mref{pro:2.6}, $(A,\cdot_{A},P,\widehat{P})$ is an admissible averaging commutative algebra.
\end{ex}

\subsection{Double constructions of averaging Frobenius commutative algebras and averaging commutative and cocommutative infinitesimal bialgebras}\

Let $(A,\cdot_{A})$ and $(A^{*},\cdot_{A^{*}})$ be commutative
algebras. If there is a commutative algebra structure $(A\oplus A^{*},\cdot_{d})$ on
$A\oplus A^{*}$ which contains $(A,\cdot_{A})$ and
$(A^{*},\cdot_{A^{*}})$ as subalgebras, and if the natural nondegenerate symmetric bilinear form
$\mathcal{B}_{d}$ given by \meqref{eq:bfds} is invariant on
$(A\oplus A^{*},\cdot_{d})$, then we call $\big( (A\oplus
A^{*},\cdot_{d},\mathcal{B}_{d}),(A,\cdot_{A}),(A^{*},\cdot_{A^{*}})
\big)$ a {\bf double construction of Frobenius commutative
algebra} \mcite{Bai2010}. In this case, the multiplication
$\cdot_{d}$ on $A\oplus A^{*}$ is given by
\vsa
\begin{equation}\mlabel{eq:commassomul}
(x+a^{*})\cdot_{d}(y+b^{*})=x\cdot_{A}y+\mathcal{L}^{*}_{\cdot_{A^{*}}}(b^{*})x+\mathcal{L}^{*}_{\cdot_{A^{*}}}(a^{*})y+
a^{*}\cdot_{A^{*}}b^{*}+\mathcal{L}^{*}_{\cdot_{A}}(y)a^{*}+\mathcal{L}^{*}_{\cdot_{A}}(x)b^{*}
\vsa
\end{equation}
for all $x,y\in A, a^{*},b^{*}\in A^{*}$.

\begin{defi}\mlabel{defi:aFca}
Let $\big( (A\oplus
A^{*},\cdot_{d},\mathcal{B}_{d}),(A,\cdot_{A}),(A^{*},\cdot_{A^{*}})
\big)$ be a double construction of Frobenius commutative algebra.
Let $P:A\rightarrow A$ be an averaging operator on
$(A,\cdot_{A})$ and $Q^{*}:A^{*}\rightarrow A^{*}$ be an averaging
operator on $(A^{*},\cdot_{A^{*}})$. If $P+Q^{*}$ is an averaging
operator on $(A\oplus A^{*},\cdot_{d})$, then we call $\big(
(A\oplus A^{*},\cdot_{d},P+Q^{*},\mathcal{B}_{d}), (A,\cdot_{A},P),(A^{*},\cdot_{A^{*}},Q^{*})
\big)$ a {\bf double construction of averaging Frobenius
commutative algebra}.
\end{defi}

Next we recall the notion of commutative and cocommutative infinitesimal bialgebras.

\begin{defi}\mcite{Bai2010}
A {\bf cocommutative (coassociative) coalgebra} is a pair $(A,\Delta)$, where $A$ is
a vector space and $\Delta:A\rightarrow A\otimes A$ is a
co-multiplication such that the following equations hold.
\vsa
\begin{equation*}
\Delta=\tau\Delta,\;(\Delta\otimes\mathrm{id})\Delta=(\mathrm{id}\otimes\Delta)\Delta.
\vsa
\end{equation*}
A {\bf commutative and cocommutative infinitesimal
bialgebra}
is a triple $(A,\cdot_{A},\Delta)$ such that $(A,\cdot_{A})$ is a
commutative algebra, $(A,\Delta)$ is a cocommutative coalgebra and the
following equation holds.
\vsa
\begin{equation}\mlabel{eq:bib}
\Delta(x\cdot_{A}y)=\big(\mathcal{L}_{\cdot_{A}}(x)\otimes\mathrm{id}\big)\Delta(y)+\big(\mathrm{id}\otimes\mathcal{L}_{\cdot_{A}}(y)\big)\Delta(x),\;\forall x,y\in A.
\vsa
\end{equation}
\end{defi}

\begin{thm}\mcite{Bai2010}
Let $(A,\cdot_{A})$ be a commutative algebra. Suppose that there is a commutative algebra structure $(A^{*},\cdot_{A^{*}})$ on the dual space $A^{*}$ and $\Delta:A\rightarrow A\otimes A$ is the linear dual of $\cdot_{A^{*}}$, that is,
\vsa
\begin{equation*}
\langle\Delta(x), a^{*}\otimes b^{*}\rangle=\langle x, \Delta^{*}(a^{*}\otimes b^{*})\rangle=\langle x, a^{*}\cdot_{A^{*}}b^{*}\rangle,\;\forall x\in A, a^{*},b^{*}\in A^{*}.
\vsa
\end{equation*}
Then there is a double construction of Frobenius commutative algebra $\big( (A\oplus A^{*},\cdot_{d},\mathcal{B}_{d}),(A$,
$\cdot_{A})$,
$(A^{*},\cdot_{A^{*}}) \big)$ if and only if $(A,\cdot_{A},\Delta)$ is a commutative and cocommutative infinitesimal bialgebra.
\end{thm}

\begin{defi}\mlabel{defi:ave com ASI bialgebra}
An {\bf averaging commutative and cocommutative
infinitesimal bialgebra}
is a vector space $A$
together with linear maps
\vsa
\begin{equation*}
\cdot_{A}:A\otimes A\rightarrow A,\; \Delta:A\rightarrow A\otimes A, \; P,Q:A\rightarrow A
\vsa
\end{equation*}
such that the following conditions are satisfied.
\begin{enumerate}
\item The triple $(A,\cdot_{A},\Delta)$ is a commutative and cocommutative infinitesimal bialgebra.
\item The quadruple $(A,\cdot_{A},P,Q)$ is an admissible averaging commutative algebra.
\item The following equalities hold.
\vsb
\begin{eqnarray}
    (Q\otimes Q)\Delta(x)&=&(Q\otimes\mathrm{id})\Delta\big(Q(x)\big),\mlabel{eq:aoco1}\\
    (Q\otimes P)\Delta(x)&=&(Q\otimes\mathrm{id})\Delta\big(P(x)\big)=(\mathrm{id}\otimes P)\Delta\big(P(x)\big),\;\forall x\in A. \mlabel{eq:aoco2}
\vsa
\end{eqnarray}
\end{enumerate}
We denote it by $(A,\cdot_{A},\Delta,P,Q)$.
\end{defi}

\begin{thm}\mlabel{thm:2.11}
Let $\big( (A\oplus
A^{*},\cdot_{d},\mathcal{B}_{d}),(A,\cdot_{A}),(A^{*},\cdot_{A^{*}})
\big)$ be a double construction of Frobenius commutative algebra,
and let $(A,\cdot_{A},\Delta)$ be the corresponding commutative and cocommutative infinitesimal bialgebra. Let $P, Q:A\to A$ be linear maps. Then the following statements are
equivalent.
\begin{enumerate}
\item\mlabel{thm:ao1} $\big( (A\oplus A^{*},\cdot_{d},P+Q^{*},\mathcal{B}_{d}),(A,\cdot_{A},P),(A^{*},\cdot_{A^{*}},Q^{*}) \big)$ is a double construction of averaging Frobenius commutative algebra.
\item\mlabel{thm:ao2} $(A,\cdot_{A},P,Q)$ and $(A^{*},\cdot_{A^{*}},Q^{*},P^{*})$ are admissible averaging commutative algebras, that is, \meqref{eq:aver op}, \meqref{eq:ao pair} and the following equations hold:
\vsb
\begin{eqnarray}
    &&Q^{*}(a^{*})\cdot_{A^{*}} Q^{*}(b^{*})=Q^{*}\big(Q^{*}(a^{*})\cdot_{A^{*}} b^{*}\big),\mlabel{eq:mp ao1}\\
    && Q^{*}(a^{*})\cdot_{A^{*}} P^{*}(b^{*})=P^{*}\big( Q^{*}(a^{*})\cdot_{A^{*}}b^{*}\big)=P^{*}\big(a^{*}\cdot_{A^{*}}P^{*}(b^{*})\big),\;\forall a^{*},b^{*}\in A^{*}.\mlabel{eq:mp ao2}
    \vsa
\end{eqnarray}
\item\mlabel{thm:ao3} $(A,\cdot_{A},\Delta,P,Q)$ is an averaging commutative and cocommutative infinitesimal bialgebra, that is, $(A,\cdot_{A},P,Q)$ is an admissible
averaging commutative algebra and satisfies
\meqref{eq:aoco1} and \meqref{eq:aoco2}.
\end{enumerate}
\end{thm}
\begin{proof}
(\mref{thm:ao1})$\Longleftrightarrow$(\mref{thm:ao2}). Let $x,y\in
A, a^{*},b^{*}\in A^{*}$. Then we have
\vsb
\begin{eqnarray*}
    &&(P+Q^{*})(x+a^{*})\cdot_{d}(P+Q^{*})(y+b^{*})\\
    &&\overset{\meqref{eq:commassomul}}{=}P(x)\cdot_{A} P(y)+\mathcal{L}^{*}_{\cdot_{A^{*}}}\big(Q^{*}(a^{*})\big)P(y)+\mathcal{L}^{*}_{\cdot_{A^{*}}}\big(Q^{*}(b^{*})\big)P(x)\\
    &&\ \
    +Q^{*}(a^{*})\cdot_{A^{*}}Q^{*}(b^{*})+\mathcal{L}^{*}_{\cdot_{A}}\big(P(x)\big)Q^{*}(b^{*})+\mathcal{L}^{*}_{\cdot_{A}}\big(P(y)\big)Q^{*}(a^{*}),\\
    &&(P+Q^{*})\big( (P+Q^{*})(x+a^{*})\cdot_{d} (y+b^{*})\big)\\
    &&\overset{\meqref{eq:commassomul}}{=}P\big(P(x)\cdot_{A}y\big)+P\big(\mathcal{L}^{*}_{\cdot_{A^{*}}}(b^{*})P(x)\big)+P\Big(\mathcal{L}^{*}_{\cdot_{A^{*}}}\big(Q^{*}(a^{*})\big)y\Big)\\
    &&\ \ +Q^{*}\big(Q^{*}(a^{*})\cdot_{A^{*}}b^{*}\big)+Q^{*}\Big(\mathcal{L}^{*}_{\cdot_{A}}\big(P(x)\big)b^{*}\Big)+Q^{*}\big(\mathcal{L}^{*}_{\cdot_{A}}(y)Q^{*}(a^{*})\big).
\end{eqnarray*}Hence $P+Q^{*}$ is an averaging operator on $(A\oplus
A^{*},\cdot_{d})$ if and only if \meqref{eq:aver op}, \meqref{eq:mp
ao1} and the following equations hold.
\vsb
\begin{eqnarray}
&&\mathcal{L}^{*}_{\cdot_{A^{*}}}\big(Q^{*}(a^{*})\big)P(x)=P\Big(\mathcal{L}^{*}_{\cdot_{A^{*}}}\big(Q^{*}(a^{*})\big)x\Big)=P\big(\mathcal{L}^{*}_{\cdot_{A^{*}}}(a^{*})P(x)\big),\mlabel{eq:pq1}\\
&&\mathcal{L}^{*}_{\cdot_{A}}\big(P(x)\big)Q^{*}(a^{*})=Q^{*}\Big(\mathcal{L}^{*}_{\cdot_{A}}\big(P(x)\big)a^{*}\Big)=Q^{*}\big(\mathcal{L}^{*}_{\cdot_{A}}(x)Q^{*}(a^{*})\big).\mlabel{eq:pq3}
\end{eqnarray}
By a  direct verification, we have
\vsc
$$ \meqref{eq:pq1}\Longleftrightarrow \meqref{eq:mp ao2},\;   \meqref{eq:pq3}\Longleftrightarrow \meqref{eq:ao pair}.
\vsa
$$
Hence Item~(\mref{thm:ao1}) holds if and only if
Item~(\mref{thm:ao2}) holds.

(\mref{thm:ao2})$\Longleftrightarrow$(\mref{thm:ao3}).
 For all $x\in
A, a^{*},b^{*}\in A^{*}$, we have
\vsb
\begin{eqnarray*}
&&\langle Q^{*}(a^{*})\cdot_{A^{*}} Q^{*}(b^{*}),x\rangle=\langle \Delta^{*}(Q^{*}\otimes Q^{*})a^{*}\otimes b^{*},x\rangle=\langle a^{*}\otimes b^{*},(Q\otimes Q)\Delta(x)\rangle,\\
&&\langle Q^{*}\big(Q^{*}(a^{*})\cdot_{A^{*}} b^{*}\big),x\rangle=\langle Q^{*}\big(\Delta^{*}(Q^{*}\otimes\mathrm{id})a^{*}\otimes b^{*}\big),x\rangle=\langle a^{*}\otimes b^{*},(Q\otimes\mathrm{id})\Delta\big(Q(x)\big)\rangle.
\vsa
\end{eqnarray*}
Hence \meqref{eq:mp ao1} holds if and only if \meqref{eq:aoco1}
holds. Similarly, \meqref{eq:mp ao2} holds if and only if
\meqref{eq:aoco2} holds. Therefore Item~(\mref{thm:ao2}) holds if
and only if Item~(\mref{thm:ao3}) holds.
\vsd\end{proof}

\section{Symmetric left-invariant bilinear forms on perm algebras, a new splitting of perm algebras and \sdpps}\mlabel{sec:4}

We  first study nondegenerate symmetric left-invariant bilinear forms on perm
algebras,  naturally arising from symmetric Frobenius commutative
algebras with averaging operators. Then we introduce the notion of
\dpps as a new splitting of the operation of perm algebras. In
particular, the notion of \sdpps which are \dpps with one of the
operations being commutative is introduced as the underlying
algebra structures of nondegenerate symmetric
left-invariant bilinear forms on perm algebras. We also study
quadratic \sdpps, which equivalently give rise to nondegenerate
symmetric left-invariant bilinear forms on perm algebras.

\subsection{Nondegenerate symmetric left-invariant bilinear forms on perm algebras} \mlabel{sec2.1}\

\begin{pro}\mlabel{pro:2.1}
Let $(A,\cdot_{A})$ be a commutative algebra with an
averaging operator $P:A\rightarrow A$, and let $(A,\circ_A)$ be the induced perm algebra $(A,\circ_{A})$ given in ~\meqref{eq:perm from aver op}.
If there is an invariant bilinear form $\mathcal{B}$ on
$(A,\cdot_{A})$, then $\mathcal{B}$ is  left-invariant on
$(A,\circ_{A})$ in the sense that
\begin{equation}\mlabel{eq:li}
\mathcal{B}(x\circ_{A} y,z)=\mathcal{B}(y,x\circ_{A} z),\;\forall x,y,z\in A.
\end{equation}
\end{pro}

\begin{proof}
For all $x,y,z\in A$, we have
\begin{equation*}
\mathcal{B}(x\circ _{A}y,z)=\mathcal{B}\big(P (x)\cdot
_{A}y,z\big)= \mathcal{B}\big(y,P (x)\cdot
_{A}z\big)=\mathcal{B}(y,x\circ _{A}z).
\end{equation*}
Hence the conclusion follows.
\end{proof}

\begin{ex}\mlabel{ex:2.2}
Under the assumptions in Example \mref{pre-ex:2.2}, the induced perm algebra $(A\oplus V,\circ_{d})$ is given by
\begin{equation}\mlabel{eq:2.2}
(x+u)\circ_{d}(y+v)=P(x+u)\cdot_{d}(y+v)=x\cdot_{A} y+\mu(x)v,\;\forall x,y\in A, u,v\in V.
\end{equation}
When $V=A^{*}, \mu=\mathcal{L}^{*}_{\cdot_{A}}$,
we have
\begin{equation}\mlabel{eq:11}
(x+a^{*})\circ_{d}(y+b^{*})=x\cdot_{A} y+\mathcal{L}^{*}_{\cdot_{A}}(x)b^{*},\;\forall x,y\in A, a^{*},b^{*}\in A^{*},
\end{equation}
and by Proposition \mref{pro:2.1}, $\mathcal{B}_{d}$ is left-invariant on  $(A\oplus A^{*},\circ_{d})$.
\end{ex}

\begin{ex}\mlabel{ex:2.3}
Under the assumptions in Example \mref{pre-ex:2.3}, the induced perm algebra $(A\otimes A,\circ)$ is given by
\begin{equation}\mlabel{eq:ex:2.3}
(x\otimes y)\circ(z\otimes w)=P(x\otimes y)\cdot (z\otimes
w)=x\cdot_{A} z\otimes y\cdot_{A} w+y\cdot_{A} z\otimes x\cdot_{A}
w, \quad \forall x,y,z,w\in A.
\end{equation}
By Proposition \mref{pro:2.1}, $\mathcal{B}'$ is also left-invariant on $(A\otimes A,\circ)$.
\end{ex}

\begin{ex}\mlabel{ex:multiplication2}
Under the assumptions in Example \mref{ex:multiplication}, the
nonzero products of the induced perm algebra $(A,\circ_{A})$ are
given as follows.
\begin{eqnarray}
e_{1}\circ_{A} e_{1}=e_{1},\; e_{1}\circ_{A} e_{2}=e_{2}\circ_{A} e_{1}=e_{2},\; e_{1}\circ_{A} e_{3}=e_{2}\circ_{A} e_{4}=e_{3},\; e_{1}\circ_{A} e_{4}=e_{4}.
\end{eqnarray}
Moreover, the nondegenerate symmetric bilinear form $\mathcal{B}$ given by
\meqref{eq:10.0} is left-invariant on $(A,\circ_{A})$.
\end{ex}

Recall that a \textbf{commutative $2$-cocycle} \mcite{Dzh2010} on a Lie algebra $(\mathfrak{%
g},[-,-]_{\mathfrak{g}})$ is a symmetric bilinear form on $\mathfrak{g}$
such that
\begin{equation}\mlabel{eq:c2c}
\mathcal{B}([x,y]_{\mathfrak{g}},z)+\mathcal{B}([y,z]_{\mathfrak{g}},x)+%
\mathcal{B}([z,x]_{\mathfrak{g}},y)=0,\;\forall x,y,z\in \mathfrak{g}.
\end{equation}
Commutative $2$-cocycles appear in the study of non-associative algebras
satisfying certain antisymmetric identities \mcite{Dzh2009}, and also in the
description of the second cohomology of current Lie algebras \mcite{Zus}.

Let $(A,\circ_{A})$ be a perm algebra. Then obviously the
multiplication $[-,-]_A$ on $A$ defined by
\begin{equation}
[x,y]_A=x\circ_A y-y\circ_A x,\;\;\forall x,y\in A,
\end{equation}
makes $(A,[-,-]_A)$ into a Lie algebra, called the {\bf
sub-adjacent Lie algebra} of $(A,\circ_A)$.

By a straightforward checking, we have
\begin{pro}\mlabel{pro:4.6}
Let $(A,\circ_{A})$ be a perm algebra and $(A$,
$[-,-]_{A})$ be the sub-adjacent Lie algebra.
If $\mathcal{B}$ is a symmetric left-invariant bilinear form on $(A,\circ_{A})$, then $\mathcal{B}$ is a commutative $2$-cocycle on $(A,[-,-]_{A})$.
\end{pro}

\subsection{A new splitting of perm algebras and \dpps}\

First we recall the notion of representations of perm algebras.

\begin{defi}\mcite{Hou,LZB}
A \textbf{representation} of a perm algebra $(A,\circ_{A})$ is a
triple $(l,r,V)$, in which $V$ is a vector space, and
$l,r:A\rightarrow\mathrm{End}_{\mathbb K}(V)$ are linear maps satisfying
\begin{eqnarray}
&&l(x\circ_{A} y)v=l(x)l(y)v=l(y)l(x)v,\mlabel{eq:rep1}\\
&&r(x\circ_{A} y)v=r(y)r(x)v=r(y)l(x)v=l(x)r(y)v,\;\forall x,y\in A, v\in V .\mlabel{eq:rep2}
\end{eqnarray}
Two representations $(l_{1},r_{1},V_{1})$ and $(l_{2},r_{2},V_{2})$ of $(A,\circ_{A})$ are called \textbf{equivalent} if there exists a linear isomorphism $\phi:V_{1}\rightarrow V_{2}$ such that
\begin{equation}\mlabel{eq:eq perm rep}
\phi l_{1}(x)=l_{2}(x)\phi,\; \phi r_{1}(x)=r_{2}(x)\phi,\;\;\forall x\in A.
\end{equation}
\end{defi}

In fact, for a vector space $V$ and linear maps $l,r:A\rightarrow
\mathrm{End}_{\mathbb K}(V)$, the triple $(l,r,V)$ is a
representation of the perm algebra $ (A,\circ_{A})$ if and only if
there is a perm algebra structure  on the direct sum $A\oplus V$
of vector spaces  given by
\begin{equation}
(x+u)\circ_{d} (y+v)=x\circ _{A}y+l(x)v+r(y)u,\;\;\forall x,y\in
A, u,v\in V. \mlabel{eq:sd perm}
\end{equation}
We denote the perm algebra structure on $A\oplus V$  by $A\ltimes
_{l,r}V$. Hence we can let $A\ltimes _{\mu,0}V$ denote the perm algebra given in
\meqref{eq:2.2}.

\begin{ex}
Let $(A,\circ_{A})$ be a perm algebra.
Then $(\mathcal{L}_{\circ_{A}},\mathcal{R}_{\circ_{A}},A)$ is a representation of $(A,\circ_{A})$, which is called the \textbf{adjoint representation} of $(A,\circ_{A})$.
\end{ex}

Now we introduce the notion of \dpps as a new type of splitting of perm
algebras.

\begin{defi}\mlabel{defi:generic pre-perm algebra2}
    Let $A$ be a vector space with multiplications $\triangleright_{A},\triangleleft_{A}:A\otimes A\rightarrow A$.
   Define a multiplication
    $\circ_{A}:A\otimes A\rightarrow A$ by
    \begin{equation}
    \circ_A:= \triangleright_{A}+\triangleleft_{A}
    \mlabel{eq:dpp2}
    \end{equation}
as in~\meqref{eq:Dpp}. Define $\bullet_A: A\otimes A\rightarrow A$
by
\begin{equation}\mlabel{eq:sum}
    x\bullet_{A}y:=x\triangleright_{A}y +y\triangleleft_{A}x,\;\;\forall x,y\in     A.
\end{equation}
If $(A,\circ_{A})$ is a perm algebra and satisfies the equalities
    \begin{eqnarray}
        &&(x\circ_{A}y)\bullet_{A}z=x\bullet_{A}(y\bullet_{A}z)=y\bullet_{A}(x\bullet_{A}z),\mlabel{eq:gppa2,2}\\
        &&z\triangleleft_{A}(x\circ_{A}y)=-(z\triangleleft_{A}y)\triangleleft_{A}x=x\bullet_{A}(z\triangleleft_{A}y)=(x\bullet_{A} z)\triangleleft_{A}y,\;\forall x,y,z\in A,\mlabel{eq:gppa3,2}
    \end{eqnarray}
then we call $(A,\triangleright_{A},\triangleleft_{A})$ an {\bf
\dpp} (with apre in short for averaging-pre).
\end{defi}

\begin{pro}\mlabel{pro:SPA}\mlabel{pro:1548}
Let $A$ be a vector space with the multiplications $\triangleright_{A}
    ,\triangleleft_{A}:A\otimes A\rightarrow A$.
    Define multiplications $\circ_{A},\bullet_{A}:A\otimes A\rightarrow A$ by
    \meqref{eq:dpp2} and \meqref{eq:sum} respectively.
The following statements are equivalent.
\begin{enumerate}
\item $(A,\triangleright_{A},\triangleleft_{A})$ is an \dpp;
\mlabel{i:spa1}
\item
 $(A,\circ_{A})$ is a perm algebra of which $(  \mathcal{L}^{*}_{\bullet_{A}}=\mathcal{L}^{*}_{\triangleright_{A}}+\mathcal{R}^{*}_{\triangleleft_{A}},-\mathcal{R}^{*}_{\triangleleft_{A}},A^{*})$ is a representation;
\mlabel{i:spa2}
\item
 $(A,\circ_{A})$ is a perm algebra of which $(\mathcal{L}_{\bullet_{A}}=\mathcal{L}_{\triangleright_{A}}+\mathcal{R}_{\triangleleft_{A}},\mathcal{L}_{\bullet_{A}}+\mathcal{R}_{\triangleleft_{A}},A)$ is a representation.
 \mlabel{i:spa3}
\end{enumerate}
\end{pro}
\begin{proof}
\meqref{i:spa1} $\Longleftrightarrow$ \meqref{i:spa2}. For all
$x,y,z\in A, a^{*}\in A^{*}$, we have
\begin{eqnarray*}
&&  \langle \mathcal{L}^{*}_{\bullet_{A}}(x\circ_{A}y)a^{*},z\rangle=\langle a^{*},(x\circ_{A}y)\bullet_{A}z\rangle,\\
&&\langle \mathcal{L}^{*}_{\bullet_{A}}(x)\mathcal{L}^{*}_{\bullet_{A}}(y)a^{*},z\rangle=\langle a^{*},y\bullet_{A}(x\bullet_{A}z)\rangle,\\
&&\langle \mathcal{L}^{*}_{\bullet_{A}}(y)\mathcal{L}^{*}_{\bullet_{A}}(x)a^{*},z\rangle=\langle a^{*},x\bullet_{A}(y\bullet_{A}z)\rangle.
\end{eqnarray*}
Thus \eqref{eq:gppa2,2} holds if and only if \eqref{eq:rep1} holds for $l=\mathcal{L}^{*}_{\bullet_{A}}$.
Similarly, \eqref{eq:gppa3,2} holds if and only if \eqref{eq:rep2} holds for $l=\mathcal{L}^{*}_{\bullet_{A}}, r=-\mathcal{R}^{*}_{\triangleleft_{A}}$.
Hence the equivalence follows.

\smallskip

\noindent \meqref{i:spa2} $\Longleftrightarrow$ \meqref{i:spa3}.
By \mcite{Hou,LZB}, $(l,r,V)$ is a representation of a perm
algebra $(A,\circ_{A})$ if and only if  $(l^{*}, l^{*}-r^{*},
V^{*})$ is a representation of $(A,\circ_{A})$. Hence the
conclusion follows by taking $l=\mathcal{L}_{\bullet_{A}},
r=\mathcal{L}_{\bullet_{A}}+\mathcal{R}_{\triangleleft_{A}},V=A$.
\end{proof}

\begin{lem}\mlabel{lem:312}
Let $(A,\circ_{A})$ be a perm algebra, $V$ be a vector space and
$l,r:A\rightarrow \mathrm{End}_{\mathbb K}(V)$ be linear maps.
Set linear maps
$\widetilde{l},\widetilde{r}:A\rightarrow \mathrm{End}_{\mathbb K}(V)$ by $\widetilde{l}=2l-r$ and $ \widetilde{r}=r-l$.
Then $(l,r,V)$ is a representation of $(A,\circ_{A})$ if and only
if the following equations hold for the triple $(\widetilde l,\widetilde r,V)$:
\begin{eqnarray}
    &&\widetilde l(y)\widetilde r(x)v+2\widetilde r(y)\widetilde r(x)v=0,\mlabel{equiv rep1}\\
    &&\widetilde r(x\circ_{A} y)v=-\widetilde r(x)\widetilde r(y)v=\widetilde r(y)(\widetilde l+\widetilde r)(x)v,\mlabel{equiv rep2}\\
    &&\widetilde l(x\circ_{A}y)v=(\widetilde l+\widetilde r)(x)\widetilde l(y)v=\widetilde l(y)(\widetilde l+\widetilde r)(x)v,\quad\forall x,y\in A, v\in V.\mlabel{equiv rep3}
\end{eqnarray}
\end{lem}
\begin{proof}
    It follows from a straightforward verification.
\end{proof}

\begin{rmk} Note that for a representation $(l,r,V)$ of a perm
algebra $(A,\circ_{A})$, $(\widetilde l,\widetilde r,V)$ might not
be a representation of $(A,\circ_{A})$. Moreover, let
$(A,\triangleright_{A},\triangleleft_{A})$ be an \dpp and define
multiplications $\circ_{A},\bullet_{A}:A\otimes A\rightarrow A$ by
\meqref{eq:dpp2} and \meqref{eq:sum} respectively. In the case
that $l=\mathcal{L}_{\bullet_{A}}$ and
$r=\mathcal{L}_{\bullet_{A}}+\mathcal{R}_{\triangleleft_{A}}$, we
have $\widetilde l=\mathcal{L}_{\triangleright_{A}}, \widetilde
r=\mathcal{R}_{\triangleleft_{A}}$. So Proposition~\ref{pro:1548}
can be rewritten in terms of $\mathcal{L}_{\triangleright_{A}},
\mathcal{R}_{\triangleleft_{A}}$.
\end{rmk}

Hence there is the following equivalent characterization of
\dpps.

\begin{pro}\mlabel{pdef:aa}
Let $A$ be a vector space with the multiplications
$\triangleright_{A},\triangleleft_{A}:A\otimes A\rightarrow A$.
Then $(A,\triangleright_{A},\triangleleft_{A})$ is an \dpp if and
only if, for all $x, y, z\in A$, the following equations hold.
\begin{eqnarray}
&&x\triangleright_{A}(y\triangleright_{A}z+y\triangleleft_{A}z)+x\triangleleft_{A}(y\triangleright_{A}z+y\triangleleft_{A}z)\nonumber\\
&&=(x\triangleright_{A}y+x\triangleleft_{A} y)\triangleright_{A}z+(x\triangleright_{A}y+x\triangleleft_{A} y)\triangleleft_{A}z\nonumber\\
&&=(y\triangleright_{A}x+y\triangleleft_{A} x)\triangleright_{A}z+(y\triangleright_{A}x+y\triangleleft_{A} x)\triangleleft_{A}z,\mlabel{eq:dpp1.2}\\
&&x\triangleright_{A}(y\triangleleft_{A}z)+2(y\triangleleft_{A}z)\triangleleft_{A}x=0,\mlabel{eq:dpp1.1}\\
&&x\triangleleft_{A}(y\triangleright_{A}z+y\triangleleft_{A}z)=-(x\triangleleft_{A}z)\triangleleft_{A}y=(y\triangleright_{A}x+x\triangleleft_{A}y)\triangleleft_{A}z,\mlabel{eq:dpp1.3}\\
&&(x\triangleright_{A}y+x\triangleleft_{A}y)\triangleright_{A}z=y\triangleright_{A}(x\triangleright_{A}z+z\triangleleft_{A}x)=x\triangleright_{A}(y\triangleright_{A}z)+(y\triangleright_{A}z)\triangleleft_{A}x.
\mlabel{eq:dpp1.4}
\end{eqnarray}
In this case, $(A,\circ_{A})$ with
$\circ_A$ defined by \meqref{eq:dpp2} is a perm algebra, called
the {\bf associated perm algebra} of
$(A,\triangleright_{A},\triangleleft_{A})$. Correspondingly,
$(A,\triangleright_{A},\triangleleft_{A})$ is called a {\bf
compatible \dpp} structure on $(A,\circ_{A})$.
\end{pro}

\begin{proof}
Suppose that $(A,\triangleright_{A},\triangleleft_{A})$ is an
\dpp. Define a multiplication $\bullet_{A}:A\otimes
A\rightarrow A$ by \meqref{eq:sum}. By Proposition
\mref{pro:1548} and Lemma \mref{lem:312}, the triple
$(\widetilde{\mathcal{L}_{\bullet_{A}}},\widetilde{\mathcal{L}_{\bullet_{A}}+\mathcal{R}_{\triangleleft_{A}}},A)=(
\mathcal{L}_{\triangleright_{A}},\mathcal{R}_{\triangleleft_{A}},A)$
satisfies \meqref{equiv rep1}--\eqref{equiv rep3}, that is,
\meqref{eq:dpp1.1}--\eqref{eq:dpp1.4} hold respectively. 
Moreover, since $(A,\circ_{A})$ is a perm algebra, \meqref{eq:dpp1.2} holds.

Conversely, a similar argument shows that, if
\meqref{eq:dpp1.2}--\eqref{eq:dpp1.4} hold, then
$(A,\triangleright_{A},\triangleleft_{A})$ is an \dpp.
\end{proof}

\begin{defi}
    Let $(l,r,V)$ be a representation of a perm algebra $(A,\circ_{A})$.
   If a linear map $T:V^{*}\rightarrow A$ satisfies the
   equation
        \begin{equation}
            T(u^{*})\circ_{A}T(v^{*})=T\Big( (l^{*}+r^{*})\big(T(u^{*})\big)v^{*}-r^{*}\big(T(v^{*})\big)u^{*}
            \Big),\;\forall u^{*},v^{*}\in V^{*},
        \end{equation}
        then we call $T$ a {\bf \ddop} of $(A,\circ_{A})$ associated to $(l,r,V)$.
       In particular, a \ddop is called {\bf strong} if there exists a perm algebra structure on $V^{*}$ given by
        \begin{equation}\mlabel{eq:circ}
            u^{*}\circ_{V^{*}}v^{*}=(l^{*}+r^{*})\big(T(u^{*})\big)v^{*}-r^{*}\big(T(v^{*})\big)u^{*},\;\;\forall
            u^*, v^*\in V^*.
        \end{equation}
\end{defi}

The notion of $\mathcal O$-operators was
first introduced on Lie algebras as a generalization of the
classical Yang-Baxter equation \mcite{Ku} and then defined for
other algebra structures.  They are also called relative
Rota-Baxter operators or generalized Rota-Baxter operators. There
are other generalizations such as anti-$\mathcal O$-operators
\mcite{GLB,LB2022}. On the other hand, for the above notion, there
is an equivalent characterization in terms of linear maps from the
representation spaces themselves to the underlying vector spaces
of the algebras. Explicitly, for a representation $(l,r,V)$ of a
perm algebra $(A,\circ_{A})$, a linear map $T:V\rightarrow A$ is
called an {\bf \dop} of $(A,\circ_{A})$ associated to $( l,r,V)$
if
        \begin{eqnarray}
            (Tu)\circ_{A}(Tv)=T\big((2l-r)(Tu)v+(r-l)(Tv)u\big)
            =T\big(\widetilde{l}(Tu)v+\widetilde{r}(Tv)u\big),\;\forall u,v\in V.\mlabel{eq:10}
        \end{eqnarray}
In particular, an \dop $T$ is called {\bf strong} if
        $(V,\circ_{V})$ given by
        \begin{equation}
            u\circ_{V}v=(2l-r)(Tu)v+(r-l)(Tv)u
            =\widetilde{l}(Tu)v+\widetilde{r}(Tv)u,\;\forall u,v\in V
        \end{equation}
        is a perm algebra. Note that an \dop associated to a
        representation $( l,r,V)$ is exactly an $\mathcal
        O$-operator associated to the triple $(\widetilde l,\widetilde r,V)$ in the usual sense.
        Moreover, for a linear map $T:V^{*}\rightarrow
        A$, $T$ is a \ddop of $(A,\circ_{A})$ associated to $(l,r,V)$ if and only if $T$ is an \dop of $(A,\circ_{A})$ associated to
        $(l^{*},l^{*}-r^{*},V^{*})$ and in addition, $T$ is
        strong as a \ddop  if and only if $T$ is
strong as an \dop. Hence in this sense, both the notions of an
\dop and a \ddop are variations of the notion of $\mathcal
O$-operators.

\begin{pro}\mlabel{pro:ddop}
      Let $(l,r,V)$ be a representation of a perm algebra $(A,\circ_{A})$.
Suppose that $T:V^{*}\rightarrow A$ is a \ddop of $(A,\circ_{A})$
associated to $(l,r,V)$.
    Define the multiplications $\triangleright_{V^{*}},\triangleleft_{V^{*}}:V^{*}\otimes V^{*}\rightarrow V^{*}$ by
    \begin{equation}\mlabel{eq:dual mul}
        u^{*}\triangleright_{V^{*}}v^{*}=(l^{*}+r^{*})(Tu^{*})v^{*},\;
        u^{*}\triangleleft_{V^{*}}v^{*}=-r^{*}(Tv^{*})u^{*},\;\forall u^{*},v^{*}\in V^{*}.
    \end{equation}
    Then $(V^{*},\triangleright_{V^{*}},\triangleleft_{V^{*}})$ is an \dpp if and only if $T$ is strong.
\end{pro}
\begin{proof}
Suppose that
$(V^{*},\triangleright_{V^{*}},\triangleleft_{V^{*}})$ is an \dpp.
Then the associated perm algebra is exactly $(V^*,\circ_{V^*})$
with $\circ_{V^*}$ defined by \meqref{eq:circ}. Hence $T$ is
strong.

Conversely, suppose that $T$ is strong. Let $u^{*},v^{*},w^{*}\in
V^{*}$. Rewrite \meqref{eq:rep1} and \meqref{eq:rep2} as
\begin{eqnarray}
&&l^{*}(x\circ_{A}y)u^{*}=l^{*}(y)l^{*}(x)u^{*}=l^{*}(x)l^{*}(y)u^{*},\mlabel{eq:rep dual1}\\
&&
r^{*}(x\circ_{A}y)u^{*}=r^{*}(x)r^{*}(y)u^{*}=l^{*}(x)r^{*}(y)u^{*}=r^{*}(y)l^{*}(x)
u^{*},\mlabel{eq:rep dual2}
\end{eqnarray}
for all $x,y\in A$. Moreover, define
multiplications $\circ_{V^*},\bullet_{V^*}$ by \meqref{eq:dpp2}
and \meqref{eq:sum} respectively. Then we have
\begin{equation}
        u^{*}\circ_{V^{*}}v^{*}=(l^{*}+r^{*}) (T u^{*}) v^{*}-r^{*} (Tv^{*}) u^{*},\; u^{*}\bullet_{V^{*}}v^{*}=l^{*} (Tu^{*}) v^{*}.
\end{equation}
Then $(V^*,\circ_{V^*})$ is a perm algebra, and we have
\begin{eqnarray*}
&&(u^{*}\circ_{V^{*}}v^{*})\bullet_{V^{*}}w^{*}=l^{*}T\big((l^{*}+r^{*})(Tu^{*})v^{*}-r^{*}(Tv^{*})u^{*}\big)w^{*}=l^{*}(Tu^{*}\circ_{A}Tv^{*})w^{*},\\
&& u^{*}\bullet_{V^{*}}(v^{*}\bullet_{V^{*}} w^{*})=l^{*}(Tu^{*})l^{*}(Tv^{*})w^{*},\\
&& v^{*}\bullet_{V^{*}}(u^{*}\bullet_{V^{*}} w^{*})=l^{*}(Tv^{*})l^{*}(Tu^{*})w^{*},\\
&&w^{*}\triangleleft_{V^{*}}(u^{*}\circ_{V^{*}}v^{*})=-r^{*}T\big((l^{*}+r^{*})(Tu^{*})v^{*}-r^{*}(Tv^{*})u^{*}\big)w^{*}=-r^{*}(Tu^{*}\circ_{A}Tv^{*})w^{*},\\
&&-(w^{*}\triangleleft_{V^{*}}v^{*})\triangleleft_{V^{*}}u^{*}=-r^{*}(Tu^{*})r^{*}(Tv^{*})w^{*},\\
&&u^{*}\bullet_{V^{*}}(w^{*}\triangleleft_{V^{*}}v^{*})=-l^{*}(Tu^{*})r^{*}(Tv^{*})w^{*},\\
&&(u^{*}\bullet_{V^{*}} w^{*})\triangleleft_{V^{*}}v^{*}=-r^{*}(Tv^{*})l^{*}(Tu^{*})w^{*}.
\end{eqnarray*}
By \meqref{eq:rep dual1} and \meqref{eq:rep dual2}, we have
\begin{eqnarray*}
     &&(u^{*}\circ_{V^{*}}v^{*})\bullet_{V^{*}}w^{*}= u^{*}\bullet_{V^{*}}(v^{*}\bullet_{V^{*}} w^{*})= v^{*}\bullet_{V^{*}}(u^{*}\bullet_{V^{*}} w^{*}),\\
    &&w^{*}\triangleleft_{V^{*}}(u^{*}\circ_{V^{*}}v^{*})= -(w^{*}\triangleleft_{V^{*}}v^{*})\triangleleft_{V^{*}}u^{*}= u^{*}\bullet_{V^{*}}(w^{*}\triangleleft_{V^{*}}v^{*})= (u^{*}\bullet_{V^{*}} w^{*})\triangleleft_{V^{*}}v^{*}.
\end{eqnarray*}
Hence $(V^{*},\triangleright_{V^{*}},\triangleleft_{V^{*}})$ is an
\dpp.
\end{proof}

\begin{pro}\mlabel{pro:dual strong}
    An invertible \ddop of a perm algebra is automatically strong.
\end{pro}
\begin{proof}
    Let $T$ be an invertible \ddop of a perm algebra $(A,\circ_{A})$ associated to $( l,r,V)$.  Then with the notations in the proof of Proposition~\mref{pro:ddop},
  we have
    \begin{eqnarray*}
        T^{-1}(Tu^{*}\circ_{A}Tv^{*})=(l^{*}+r^{*}) (T u^{*}) v^{*}-r^{*} (T v^{*}
        )u^{*}=u^{*}\circ_{V^{*}}v^{*},\;\;  \forall u^{*},v^{*}\in
        V^{*}.
    \end{eqnarray*}
Therefore for all $u^{*},v^{*},w^{*}\in V^{*}$, we have
    \begin{eqnarray*}
        &&u^{*}\circ_{V^{*}}(v^{*}\circ_{V^{*}}w^{*})=T^{-1}\big(Tu^{*}\circ_{A}T(v^{*}\circ_{V^{*}}w^{*})\big)=T^{-1}\big(Tu^{*}\circ_{A}(Tv^{*}\circ_{A}Tw^{*})\big),\\
        &&(u^{*}\circ_{V^{*}}v^{*})\circ_{V^{*}}w^{*}=T^{-1}\big(T(u^{*}\circ_{V^{*}}v^{*})\circ_{A}Tw^{*}\big)=T^{-1}\big((Tu^{*}\circ_{A}Tv^{*})\circ_{A}Tw^{*}\big),\\
        &&(v^{*}\circ_{V^{*}}u^{*})\circ_{V^{*}}w^{*}=T^{-1}\big(T(v^{*}\circ_{V^{*}}u^{*})\circ_{A}Tw^{*}\big)=T^{-1}\big((Tv^{*}\circ_{A}Tu^{*})\circ_{A}Tw^{*}\big).
    \end{eqnarray*}
    Thus $(V^{*},\circ_{V^{*}})$ is a perm algebra, and hence $T$ is strong.
\end{proof}

\begin{thm}\mlabel{thm:dual 2}
     There is a compatible \dpp structure $(A,\triangleright_{A},
    \triangleleft_{A})$ on a perm algebra $(A,\circ_{A})$
 if and only if there is an invertible \ddop $T$ of $(A,\circ_{A})$ associated to  $( l,r,V)$. In this case, the multiplications $\triangleright_{A},\triangleleft_{A}$ are respectively defined by
    \begin{equation}\mlabel{deq:pro:2.14}
        x\triangleright_{A}y=T ( l^{*}+r^{*}) (x)T^{-1}(y) ,\;
        x\triangleleft_{A}y=-T r^{*}(y)T^{-1}(x),\;\forall x,y\in A.
    \end{equation}
\end{thm}
\begin{proof}
    Suppose that $T:V^{*}\rightarrow A$ is an invertible \ddop of $(A,\circ_{A})$ associated to $( l,r,V)$. Then  there is an induced \dpp structure $(V^{*},\triangleright_{V^{*}},\triangleleft_{V^{*}})$ on $V^{*}$ given by \meqref{eq:dual mul}.
    The linear isomorphism $T$ gives an \dpp structure $(A,\triangleright_{T},\triangleleft_{T})$ on $A$ by
    \begin{eqnarray*}
        &&x\triangleright_{T}y:=T\big(T^{-1}(x)\triangleright_{V^*}T^{-1}(y)\big)\overset{\meqref{eq:dual mul}}{=}T ( l^{*}+r^{*}) (x)T^{-1}(y):=x\triangleright_{A}y,\\
        &&x\triangleleft_{T}y:=T\big(T^{-1}(x)\triangleleft_{V^*}T^{-1}(y)\big)\overset{\meqref{eq:dual mul}}{=}-T
        r^{*}(y)T^{-1}(x):=x\triangleleft_{A}y,\;\;\forall x,y\in A.
    \end{eqnarray*}
Moreover, we obviously have $x\triangleright_A y+x\triangleleft_Ay=x\circ_A y$ for
all $x,y\in A$.

Conversely, suppose that
$(A,\triangleright_{A},\triangleleft_{A})$ is an \dpp. Define
multiplications $\circ_{A},\bullet_{A}:A\otimes A\rightarrow A$ by
\meqref{eq:dpp2} and \meqref{eq:sum} respectively. Then $(
\mathcal{L}^{*}_{\bullet_{A}},-\mathcal{R}^{*}_{\triangleleft_{A}},A^{*})$
is a representation of the associated perm algebra
$(A,\circ_{A})$. For all $x,y\in A$, we have
\begin{eqnarray*}
    \mathrm{id}(x)\circ_{A}\mathrm{id}(y)
    =x\triangleright_{A}y+x\triangleleft_{A}y
    =\mathrm{id}\Big( ( \mathcal{L}^{*}_{\bullet_{A}} -\mathcal{R}^{*}_{\triangleleft_{A}})^{*}\big(\mathrm{id}(x)\big)y-(-\mathcal{R}^{*}_{\triangleleft_{A}})^{*}\big(\mathrm{id}(y)\big)x \Big).
\end{eqnarray*}
Thus $T=\mathrm{id}\in\mathrm{Hom}_{\mathbb K}\big(
(A^{*})^{*},A \big)=\mathrm{End}_{\mathbb K}(A)$ is an invertible
\ddop of $(A,\circ_{A})$ associated to $(
\mathcal{L}^{*}_{\bullet_{A}},-\mathcal{R}^{*}_{\triangleleft_{A}},A^{*})$.
\end{proof}

Let $V$ be a vector space. Then the isomorphism ${\rm
Hom}_{\mathbb K}(V\otimes V,\mathbb K)\cong {\rm Hom}_{\mathbb
K}(V, V^*)$ identifies a bilinear form  $\mathcal{B}:V\otimes
V\rightarrow \mathbb K$ on V with a linear map $\mathcal
B^\natural:V\rightarrow V^*$ by
$$\mathcal B(u,v)=\langle \mathcal B^\natural (u),
v\rangle,\;\;\forall u,v\in V.$$ Moreover, $\mathcal B$ is
nondegenerate if and only if $\mathcal B^\natural$ is invertible.

\begin{thm}\mlabel{thm:420}
    Let $\mathcal{B}$ be a nondegenerate bilinear form on a perm algebra $(A,\circ_{A})$ satisfying
    \begin{equation}\mlabel{eq:left inv1}
        \mathcal{B}(x\circ_{A}y,z)=\mathcal{B}(y,x\circ_{A}z+z\circ_{A}x)-\mathcal{B}(x,z\circ
        _{A} y),\;\forall x,y,z\in A.
    \end{equation}
    Then there is a compatible \dpp $(A,\triangleright_{A},\triangleleft_{A})$ with  multiplications $\triangleright_{A},\triangleleft_{A}:A\otimes
    A\rightarrow A$ defined respectively by
    \begin{eqnarray}
        \mathcal{B}(x\triangleright_{A}y,z)&=&\mathcal{B}(y,x\circ_{A}z+z\circ_{A}x),\mlabel{eq:cor3}\\
        \mathcal{B}(x\triangleleft_{A}y,z)&=&-\mathcal{B}(x,z\circ_{A} y), \quad \forall x, y, z\in A.\mlabel{eq:cor4}
    \end{eqnarray}
Moreover, $(\mathcal{L}^{*}_{\bullet_{A}},-\mathcal{R}^{*}_{\triangleleft_{A}},A^{*})$ and $(\mathcal{L} _{\circ_{A}},\mathcal{R} _{\circ_{A}},A)$ are equivalent as representations of $(A,\circ_{A})$, or equivalently, $(\mathcal{L}_{\bullet_{A}},\mathcal{L}_{\bullet_{A}}+\mathcal{R}_{\triangleleft_{A}},A)$ and  $(\mathcal{L}^{*}_{\circ_{A}},\mathcal{L}^{*}_{\circ_{A}}-\mathcal{R}^{*}_{\circ_{A}},A^{*})$ are equivalent as representations
of $(A,\circ_{A})$.
\end{thm}
\begin{proof}
    Let $x,y,z\in A$ and $a^{*}=\mathcal{B}^{\natural}(x), b^{*}=\mathcal{B}^{\natural}(y)$. Then we have
        \begin{eqnarray*}
        &&\langle \mathcal{B}^{\natural}\big( \mathcal{B}^{\natural^{-1}}(a^{*})\circ_{A} \mathcal{B}^{\natural^{-1}}(b^{*}) \big),z\rangle=\langle \mathcal{B}^{\natural}(x\circ_{A}y),z\rangle=\mathcal{B}(x\circ_{A}y,z)\\
            &&=\mathcal{B}(y,x\circ_{A}z+z\circ_{A}x)-\mathcal{B}(x,z\circ_{A}y)=\langle \mathcal{B}^{\natural}(y),x\circ_{A}z+z\circ_{A}x\rangle-\langle \mathcal{B}^{\natural}(x),z\circ_{A}y\rangle\\
            &&=\langle (\mathcal{L}^{*}_{\circ_{A}}+\mathcal{R}^{*}_{\circ_{A}})\big(\mathcal{B}^{\natural^{-1}}(a^{*})\big)b^{*},z\rangle-\langle \mathcal{R}^{*}_{\circ_{A}}\big(\mathcal{B}^{\natural^{-1}}(b^{*})\big)a^{*},z\rangle.
        \end{eqnarray*}
    Hence we have
        \begin{align*}
        \mathcal{B}^{\natural^{-1}}(a^{*})\circ_{A}\mathcal{B}^{\natural^{-1}}(b^{*})=\mathcal{B}^{\natural^{-1}} \Big((\mathcal{L}^{*}_{\circ_{A}}+\mathcal{R}^{*}_{\circ_{A}})\big( \mathcal{B}^{\natural^{-1}}(a^{*})\big)b^{*}-\mathcal{R}^{*}_{\circ_{A}}\big( \mathcal{B}^{\natural^{-1}}(b^{*}) \big)a^{*}\Big).
        \end{align*}
    That is, $\mathcal{B}^{\natural^{-1}}:A^{*}\rightarrow A$ is an invertible \ddop of $(A,\circ_{A})$ associated to the adjoint representation
    $ (\mathcal{L} _{\circ_{A}}, \mathcal{R} _{\circ_{A}},A )$.
    Hence by Theorem \mref{thm:dual 2}, there are multiplications $\triangleright_{A},\triangleleft_{A}$ defined by
    \meqref{deq:pro:2.14}, where
        $T=\mathcal{B}^{\natural^{-1}}$ and $(l,r,V)= (\mathcal{L} _{\circ_{A}}, \mathcal{R} _{\circ_{A}},A
        )$,
    such that $(A,\triangleright_{A},\triangleleft_{A})$ is an \dpp.
    Explicitly, we have
    \begin{eqnarray*}
   &&x\triangleright_{A}y=\mathcal{B}^{\natural^{-1}}\big((\mathcal{L}^{*}_{\circ_{A}}+\mathcal{R}^{*}_{\circ_{A}}) (x)\mathcal{B}^{\natural}(y) \big),\;\; x\triangleleft_{A}y=-\mathcal{B}^{\natural^{-1}}\big(\mathcal{R}^{*}_{\circ_{A}}(y)\mathcal{B}^{\natural}(x) \big),
    \end{eqnarray*}
    and thus
    \begin{eqnarray*}
            &&\mathcal{B}(x\triangleright_{A}y,z)=\langle(\mathcal{L}^{*}_{\circ_{A}}+\mathcal{R}^{*}_{\circ_{A}}) (x)\mathcal{B}^{\natural}(y) ,z\rangle=\langle \mathcal{B}^{\natural}(y),x\circ_{A}z+z\circ_{A}x\rangle=\mathcal{B}\big(y,x\circ_{A}z+z\circ_{A}x\big),\\
            &&\mathcal{B}(x\triangleleft_{A}y,z)=-\langle \mathcal{R}^{*}_{\circ_{A}}(y)\mathcal{B}^{\natural}(x), z\rangle=-\langle \mathcal{B}^{\natural}(x),z\circ_{A}y\rangle=-\mathcal{B}\big(x,z\circ_{A}y\big).
    \end{eqnarray*}

Moreover, define a multiplication $\bullet_{A}:A\otimes A\rightarrow A$ by
 \meqref{eq:sum}. Again by
\meqref{eq:cor3} and \meqref{eq:cor4}, we have
        \begin{eqnarray}
            \mathcal{B}(x\bullet_{A}y ,z)=\mathcal{B}(y,x\circ_{A}z).\mlabel{eq:bf3}
        \end{eqnarray}
        Thus we have
        \begin{eqnarray*}
        &&\langle \mathcal{B}^{\natural}\big(\mathcal{L}_{\bullet_{A}}(x)y\big),z\rangle=\mathcal{B}(x\bullet_{A}y ,z)
        =\mathcal{B}(y,x\circ_{A}z)=\langle \mathcal{B}^{\natural}(y),x\circ_{A}z\rangle=\langle \mathcal{L}^{*}_{\circ_{A}}(x)\mathcal{B}^{\natural}(y),z\rangle,\\
        &&\langle \mathcal{B}^{\natural}\big(\mathcal{R}_{\triangleleft_{A}}(x)y\big),z\rangle
        =\mathcal{B}(y\triangleleft_{A}x,z)=-\mathcal{B}(y,z\circ_{A}x)=-\langle \mathcal{B}^{\natural}(y),z\circ_{A}x\rangle=
        -\langle \mathcal{R}^{*}_{\circ_{A}}(x)\mathcal{B}^{\natural}(y),z\rangle.
        \end{eqnarray*}
        Hence
        \begin{eqnarray*}
        \mathcal{B}^{\natural}\big(\mathcal{L}_{\bullet_{A}}(x)y\big)=\mathcal{L}^{*}_{\circ_{A}}(x)\mathcal{B}^{\natural}(y),\;
        \mathcal{B}^{\natural}\big((\mathcal{L}_{\bullet_{A}}+\mathcal{R}_{\triangleleft_{A}})(x)y\big)=(\mathcal{L}^{*}_{\circ_{A}}-\mathcal{R}^{*}_{\circ_{A}})(x)\mathcal{B}^{\natural}(y).
        \end{eqnarray*}
        Therefore, the bijection $\mathcal{B}^{\natural}:A\rightarrow A^{*}$ gives the equivalence between $(\mathcal{L}_{\bullet_{A}},\mathcal{L}_{\bullet_{A}}+\mathcal{R}_{\triangleleft_{A}},A)$ and  $(\mathcal{L}^{*}_{\circ_{A}},\mathcal{L}^{*}_{\circ_{A}}-\mathcal{R}^{*}_{\circ_{A}},A^{*})$ as representations of $(A,\circ_{A})$.
        Similarly by \meqref{eq:cor4} and \meqref{eq:bf3}, we have
       \begin{equation}\mlabel{eq:bf qua}
       (\mathcal{B}^{\natural})^{*}\big(\mathcal{L}_{\circ_{A}}(x)z\big)=\mathcal{L}^{*}_{\bullet_{A}}(x)(\mathcal{B}^{\natural})^{*}(z),\;
       (\mathcal{B}^{\natural})^{*}\big(\mathcal{R}_{\circ_{A}}(y)z\big)=-\mathcal{R}^{*}_{\triangleleft_{A}}(y)(\mathcal{B}^{\natural})^{*}(z).
       \end{equation}
       Therefore, the bijection $(\mathcal{B}^{\natural})^{*}:A\rightarrow A^{*}$ gives the equivalence between
        $(\mathcal{L}^{*}_{\bullet_{A}},-\mathcal{R}^{*}_{\triangleleft_{A}},A^{*})$ and $(\mathcal{L} _{\circ_{A}},\mathcal{R} _{\circ_{A}},A)$  as representations of $(A,\circ_{A})$.
\end{proof}

\begin{rmk}
In fact, since the \dpp in Theorem \mref{thm:420} is compatible
with $(A,\circ_{A})$, it is enough to retain one of
\meqref{eq:cor3} and \meqref{eq:cor4} and the other multiplication
is given from the substraction, that is,
\begin{equation*}
x\triangleleft_{A}y=x\circ_{A} y-x\triangleright_{A}y,\;\text{or}\;
x\triangleright_{A}y=x\circ_{A} y-x\triangleleft_{A}y,\;\forall x,y\in A.
\end{equation*}
\end{rmk}

The converse of Theorem~\mref{thm:420} is also true.

\begin{cor}\mlabel{cor:inv bf}
Let $(A,\triangleright_{A},\triangleleft_{A})$ be a compatible
\dpp structure on a perm algebra $(A,\circ_{A})$. Define a
multiplication $\bullet_{A}:A\otimes A\rightarrow A$ by
\meqref{eq:sum}. If $(\mathcal{L} _{\circ_{A}},\mathcal{R}
_{\circ_{A}},A)$ and $(
\mathcal{L}^{*}_{\bullet_{A}},-\mathcal{R}^{*}_{\triangleleft_{A}},A^{*})$
are equivalent as representations of $(A,\circ_{A})$, then there
exists a nondegenerate bilinear form $\mathcal{B}$ on
$(A,\circ_{A})$ satisfying \meqref{eq:left inv1}.
\end{cor}

\begin{proof}
Let $\phi:A\rightarrow A^{*}$ be a bijection which gives the
equivalence between $(\mathcal{L} _{\circ_{A}},\mathcal{R}
_{\circ_{A}},A)$ and
$(\mathcal{L}^{*}_{\bullet_{A}},-\mathcal{R}^{*}_{\triangleleft_{A}},A^{*})$
as representations of $(A,\circ_{A})$. Then there is a
nondegenerate bilinear form $\mathcal{B}$ on $A$ given by
\begin{equation}\mlabel{eq:phi2}
    \mathcal{B}(x,y)=\langle x,\phi(y)\rangle,\;\forall x,y\in A,
\end{equation}
that is, $\phi=(\mathcal{B}^{\natural})^{*}$. Hence
\meqref{eq:bf qua} holds. Therefore, equivalently,
\meqref{eq:cor4} and \meqref{eq:bf3} hold. Hence \meqref{eq:cor3}
holds. Furthermore, by \meqref{eq:cor3} and \meqref{eq:cor4}, we
obtain \meqref{eq:left inv1}. Hence the conclusion follows.
\end{proof}

\subsection{Special \dpps}\mlabel{sec2.2}\

\begin{defi}
A {\bf \sdpp} 
 is an \dpp $(A,\triangleright_{A},\triangleleft_{A})$ in which $\triangleleft_{A}$ is commutative.
\end{defi}

Note that for a \sdpp  $(A,\triangleright_{A},\triangleleft_{A})$, the
multiplications $\circ_{A},\bullet_{A}:A\otimes A\rightarrow A$
defined respectively by \meqref{eq:dpp2} and \meqref{eq:sum}
coincide.

\begin{cor}\mlabel{cor:SDPP}
    Let $A$ be a vector space with multiplications $\triangleright_{A}
    ,\triangleleft_{A}:A\otimes A\rightarrow A$.
    Set a multiplication $\circ_{A} :A\otimes A\rightarrow A$ by
    \meqref{eq:dpp2}.
    Then the following statements are equivalent:
    \begin{enumerate}
        \item\mlabel{s1} $(A,\triangleright_{A},\triangleleft_{A})$ is a \sdpp.
        \item\mlabel{s4} The multiplication $\triangleleft_{A}$ is commutative,  $(A,\circ_{A})$ is a perm algebra and the following equation holds:
        \begin{equation}\mlabel{eq:SDPP}
            (x\circ_{A}y)\triangleleft_{A}z=x\circ_{A}(y\triangleleft_{A}z)=-x\triangleleft_{A}(y\triangleleft_{A}z),\;\forall x,y,z\in A.
        \end{equation}
        \item \mlabel{s5} The multiplication $\triangleleft_{A}$ is
        commutative, and equations \meqref{eq:dpp1.1}, \meqref{eq:dpp1.4} and the following one hold.
         \begin{equation}\mlabel{eq:dpp1.5}
         x\triangleleft_{A}(y\triangleright_{A}z+y\triangleleft_{A}z)=-(x\triangleleft_{A}z)\triangleleft_{A}y.
         \end{equation}
        \item\mlabel{s3} The multiplication $\triangleleft_{A}$ is commutative, and $(A,\circ_{A})$ is a perm algebra with a representation $(\mathcal{L}^{*}_{\circ_{A}},-\mathcal{L}^{*}_{\triangleleft_{A}},A^{*})$.
        \item\mlabel{s2} The multiplication $\triangleleft_{A}$ is commutative, and $(A,\circ_{A})$ is a perm algebra with a representation  $(\mathcal{L}_{\circ_{A}},\mathcal{L}_{\circ_{A}}+\mathcal{L}_{\triangleleft_{A}},A)$.
    \end{enumerate}
\end{cor}
\begin{proof}

(\mref{s1}) $\Longleftrightarrow$ (\mref{s4}). It follows
immediately from Definition~\mref{defi:generic pre-perm algebra2}
in the case that $\triangleleft_{A}$ is commutative.

(\mref{s1}) $\Longleftrightarrow$ (\mref{s5}). Suppose that
Item (\mref{s1}) holds. Then by Proposition~\mref{pdef:aa},
\meqref{eq:dpp1.1}, \meqref{eq:dpp1.3} and  \meqref{eq:dpp1.4}
hold, which gives Item (\mref{s5}). Conversely, suppose that
Item (\mref{s5}) holds. By the commutativity of
$\triangleleft_{A}$, \meqref{eq:dpp1.5} gives rise to
\meqref{eq:dpp1.3}. Moreover, we have
\begin{eqnarray*}
&&x\triangleright_{A}(y\triangleright_{A}z+y\triangleleft_{A}z)+x\triangleleft_{A}(y\triangleright_{A}z+y\triangleleft_{A}z)-(x\triangleright_{A}y+x\triangleleft_{A} y)\triangleright_{A}z-(x\triangleright_{A}y+x\triangleleft_{A} y)\triangleleft_{A}z\\
&&\overset{\meqref{eq:dpp1.4},\meqref{eq:dpp1.5}}{=}x\triangleright_{A}(y\triangleleft_{A}z)+x\triangleleft_{A}(y\triangleleft_{A}z)+(z\triangleleft_{A} y)\triangleleft_{A}x\overset{\meqref{eq:dpp1.1}}{=}0,\\
&&x\triangleright_{A}(y\triangleright_{A}z+y\triangleleft_{A}z)+x\triangleleft_{A}(y\triangleright_{A}z+y\triangleleft_{A}z) -(y\triangleright_{A}x+y\triangleleft_{A} x)\triangleright_{A}z-(y\triangleright_{A}x+y\triangleleft_{A} x)\triangleleft_{A}z\\
&&\overset{\meqref{eq:dpp1.5}}{=}x\triangleright_{A}(y\triangleright_{A}z+y\triangleleft_{A}z)-(y\triangleright_{A}x+y\triangleleft_{A} x)\triangleright_{A}z\overset{\meqref{eq:dpp1.4}}{=}0.
\end{eqnarray*}
By Proposition~\mref{pdef:aa} and the commutativity of
$\triangleleft_{A}$, we obtain Item (\mref{s1}).

(\mref{s1}) $\Longleftrightarrow$ (\mref{s3})  $\Longleftrightarrow$
(\mref{s2}). It follows from  Proposition \mref{pro:SPA}.
\end{proof}

There is the following relationship between admissible averaging commutative algebras and \sdpps.

\begin{pro}\mlabel{pro:com asso and SDPP}
    Let $(A,\cdot_{A},P,Q)$ be an admissible averaging commutative
    algebra.
    Let $\triangleright_{A},\triangleleft_{A}:A\otimes A\rightarrow A$ be multiplications on $A$ defined by
    \begin{equation}\mlabel{eq:com asso and SDPP}
        x\triangleright_{A}y=P(x)\cdot_{A}y+Q(x\cdot_{A}y),\;
        x\triangleleft_{A}y=-Q(x\cdot_{A}y),\;\forall x,y\in A.
    \end{equation}
    Then
    $(A,\triangleright_{A},\triangleleft_{A})$ is a \sdpp.
\end{pro}
\begin{proof}Define an operation $\circ_A$ on $A$ by
    \meqref{eq:dpp2}. Then \meqref{eq:perm from aver op} holds and hence  by
    Proposition~\mref{ex:comm aver}, $(A,\circ_{A})$ is a perm algebra.
    By the assumption,
    $(A\ltimes_{\mathcal{L}^*_{\cdot_{A}}}A^*,P+Q^*)$ is an averaging
    commutative algebra. Thus there is a perm algebra
    $(A\oplus A^*,\circ_{d})$ with $\circ_d$ defined by
    \begin{eqnarray*}
        (x+a^{*})\circ_{d}(y+b^{*})&=&(P+Q^*)(x+a^{*})\cdot_{d}(y+b^{*})\\
        &=&\big(P(x)+Q^*(a^{*})\big)\cdot_{d}(y+b^{*})\\
        &=&P(x)\cdot_{A} y+\mathcal{L}^{*}_{\cdot_{A}}\big(P(x)\big)b^{*}+\mathcal{L}^{*}_{\cdot_{A}}(y)Q^{*}(a^{*})\\
        &\overset{\meqref{eq:perm from aver
                op},\meqref{eq:com asso and SDPP}}{=}&x\circ_{A}
        y+\mathcal{L}^{*}_{\circ_{A}}(x)b^{*}-\mathcal{L}^{*}_{\triangleleft_{A}}(y)a^{*},\;\;\forall
        x,y\in A, a^*,b^*\in A^*.
    \end{eqnarray*}
    That is, $(\mathcal{L}^{*}_{\circ_{A}},-\mathcal{L}^{*}_{\triangleleft_{A}},A^{*})$ is a representation of $(A,\circ_{A})$. Hence  $(A,\triangleright_{A},\triangleleft_{A})$ is a \sdpp which is compatible with $(A,\circ_{A})$.
\end{proof}

\begin{defi}
Let $(l,r,V)$ be a representation of a perm algebra
$(A,\circ_{A})$.  A \ddop $T:V^{*}\rightarrow A$ of $(A,\circ_A)$
associated to $(l,r,V)$ is called {\bf special} if
        \begin{equation}
            r^{*}\big(T(u^{*})\big)v^{*}=r^{*}\big(T(v^{*})\big)u^{*},\;\;\forall
            u^*,v^*\in V^*.
        \end{equation}
\end{defi}

\begin{cor}\mlabel{cor:44}
\begin{enumerate}
\item \mlabel{it:11} Let $T:V\rightarrow A^*$ be a \ddop of a perm
algebra $(A,\circ_A)$ associated to a representation $(l,r,V)$.
Define the multiplications
$\triangleright_{V^{*}},\triangleleft_{V^{*}}:V^{*}\otimes V^{*}\rightarrow V^{*}$
by \meqref{eq:dual mul}. Then $(V^{*},\triangleright_{V^{*}},\triangleleft_{V^{*}})$
is a \sdpp if and only if $T$ is strong and special.

\item \mlabel{it:22}  There is a compatible \sdpp structure
$(A,\triangleright_{A},\triangleleft_{A})$ on a perm algebra $(A,\circ_{A})$
 if and only if there is an invertible special \ddop $T$ of $(A,\circ_{A})$ associated to  $( l,r,V)$. In this case, the multiplications $\triangleright_{A},\triangleleft_{A}$ are defined by
   \meqref{deq:pro:2.14}.
\end{enumerate}
\end{cor}

\begin{proof}
(\mref{it:11}). It follows directly from
Proposition~\mref{pro:ddop}.

(\mref{it:22}). It follows directly from Theorem~\mref{thm:dual 2}.
\end{proof}

\begin{lem}\mlabel{lem:transfer111} Let $\mathcal{B}$ be a bilinear form on a perm algebra
$(A,\circ_{A})$ satisfying
\begin{eqnarray}\mlabel{eq:67}
    \mathcal{B}(x\circ_{A}y,z)=\mathcal{B}(y,x\circ_{A}z)=\mathcal{B}(x\circ_{A}z,y),\;\forall x,y,z\in A.
\end{eqnarray}
Then \meqref{eq:left inv1} holds. In particular, if $\mathcal{B}$
is a symmetric left-invariant bilinear form on a perm algebra
$(A,\circ_{A})$, then \meqref{eq:left inv1} holds.
\end{lem}
\begin{proof}
It is straightforward.
\end{proof}

\begin{cor}\mlabel{cor:quadratic1}
Let $\mathcal{B}$ be a nondegenerate bilinear form on a perm
algebra $(A,\circ_{A})$ satisfying \meqref{eq:67}. Then there is a
compatible \sdpp structure
$(A,\triangleright_{A},\triangleleft_{A})$ with multiplications
$\triangleright_{A},\triangleleft_{A}:A\otimes A\rightarrow A$
given by \meqref{eq:cor3} and \meqref{eq:cor4} respectively.
 Moreover, $(  \mathcal{L}^{*}_{\circ_{A}},-\mathcal{L}^{*}_{\triangleleft_{A}},A^{*})$ and $(\mathcal{L} _{\circ_{A}},\mathcal{R} _{\circ_{A}},A)$ are equivalent as representations of $(A,\circ_{A})$.
Conversely, let $(A,\triangleright_{A},\triangleleft_{A})$ be a
compatible \sdpp structure on a perm algebra $(A,\circ_{A})$. If
 $(\mathcal{L} _{\circ_{A}},\mathcal{R} _{\circ_{A}},A)$ and $(
 \mathcal{L}^{*}_{\circ_{A}},-\mathcal{L}^{*}_{\triangleleft_{A}},A^{*})$
are
equivalent as representations of $(A,\circ_{A})$, then there
exists a nondegenerate bilinear form $\mathcal{B}$ on
$(A,\circ_{A})$ satisfying \meqref{eq:67}.
\end{cor}

\begin{proof}
By Theorem \mref{thm:420} and Lemma \mref{lem:transfer111},
$(A,\triangleright_{A},\triangleleft_{A})$ is a compatible \dpp
structure on $(A,\circ_A)$. By \meqref{eq:67} and \meqref{eq:cor4},
we have
$$\mathcal{B}(x\triangleleft_{A}y,z)=-\mathcal{B}(x,z\circ_{A}
y)=-\mathcal{B}(y, z\circ_A
x)=\mathcal{B}(y\triangleleft_{A}x,z),\;\forall x,y,z\in A.$$ By
the nondegeneracy of $\mathcal{B}$, we have $
x\triangleleft_{A}y=y\triangleleft_{A}x $ for all $x,y\in A$ and
hence $(A,\triangleright_{A},\triangleleft_{A})$ is a \sdpp.
Define a multiplication $\bullet_{A}:A\otimes A\rightarrow A$
by \meqref{eq:sum}. By Theorem \mref{thm:420},
$(\mathcal{L}^{*}_{\bullet_{A}},-\mathcal{R}^{*}_{\triangleleft_{A}},A^{*})
=(
\mathcal{L}^{*}_{\circ_{A}},-\mathcal{L}^{*}_{\triangleleft_{A}},A^{*})$
and $(\mathcal{L} _{\circ_{A}},\mathcal{R} _{\circ_{A}},A)$ are
equivalent as representations of $(A,\circ_{A})$.

Conversely, suppose that $\phi:A\rightarrow A^{*}$ is the bijection giving the equivalence between   $(\mathcal{L} _{\circ_{A}},\mathcal{R} _{\circ_{A}},$
$A)$
and $(
\mathcal{L}^{*}_{\circ_{A}},-\mathcal{L}^{*}_{\triangleleft_{A}},A^{*})$ as representations of $(A,\circ_{A})$.
 Then there is a nondegenerate bilinear form $\mathcal{B}$ on $A$ given by \meqref{eq:phi2}.
Moreover, we have
\begin{eqnarray*}
&&\mathcal{B}(x\circ_{A}y,z)=\langle x\circ_{A}y,\phi(z)\rangle=\langle y,\mathcal{L}^{*}_{\circ_{A}}(x)\phi(z)\rangle=\langle y,\phi(x\circ_{A}z)\rangle=\mathcal{B}(y,x\circ_{A}z),\\
&&\mathcal{B}(x\circ_{A}y,z)=\langle x\circ_{A}y,\phi(z)\rangle=\langle x,\mathcal{R}^{*}_{\circ_{A}}(y)\phi(z)\rangle=-\langle x,\phi(\mathcal{L}_{\triangleleft_{A}}(y)z)\rangle=-\mathcal{B}(x,y\triangleleft_{A} z),
\end{eqnarray*}
for all $x,y,z\in A$. By the commutativity of
$\triangleleft_{A}$, we obtain
$\mathcal{B}(x\circ_{A}y,z)=\mathcal{B}(x\circ_{A}z,y)$ for
all $x,y,z\in A$. Hence the conclusion follows.
\end{proof}

\begin{lem}\mlabel{lem:337}
Let $(A,\triangleright_{A},\triangleleft_{A})$ be a \sdpp and $(A,\circ_{A})$ be
the associated perm algebra. Suppose that $\mathcal{B}$ is a symmetric bilinear form on $A$ satisfying \meqref{eq:cor4}.
Then \meqref{eq:li}, \meqref{eq:cor3} and the following equation hold:
\begin{equation}\mlabel{eq:cor3.37} \mathcal{B}(x\triangleright_{A}y,z)=\mathcal{B}(x,z\triangleright_{A}y),\;\forall x,y,z\in A.
\end{equation}
\end{lem}
\begin{proof}
Let  $x,y,z\in A$. Then we have
\begin{eqnarray*}
    \mathcal{B}(x\circ_{A}y,z)-\mathcal{B}(y,x\circ_{A} z)
    =\mathcal{B}(z,x\circ_{A}y)-\mathcal{B}(y,x\circ_{A} z)
    \overset{\meqref{eq:cor4}}{=}\mathcal{B}(-z\triangleleft_{A}y+y\triangleleft_{A}z,x)=0,
\end{eqnarray*}
that is, \meqref{eq:li} holds. Moreover, we have
\begin{eqnarray*}
&&\mathcal{B}(x\triangleright_{A} y,z)=\mathcal{B}(x\circ_{A}y-x\triangleleft_{A}y,z)=\mathcal{B}(x\circ_{A}y,z)+\mathcal{B}(x,z\circ_{A}y)\\
&&=\mathcal{B}(x\circ_{A}y,z)+\mathcal{B}(z\circ_{A}y,x)\overset{\meqref{eq:li}}{=}\mathcal{B}(y,x\circ_{A}z+z\circ_{A}x),
\end{eqnarray*}
that is, \meqref{eq:cor3} holds. Note that
$\mathcal{B}(y,x\circ_{A}z+z\circ_{A}x)$ is symmetric in $x$ and
$z$. Then we have
\begin{eqnarray*}
\mathcal{B}(x\triangleright_{A} y,z)=\mathcal{B}(z\triangleright_{A} y,x)=\mathcal{B}(x,z\triangleright_{A} y ),
\end{eqnarray*}
that is, \meqref{eq:cor3.37} holds.
\end{proof}

\begin{defi} A (symmetric) bilinear form $\mathcal{B}$ on a \sdpp  $(A,\triangleright_{A},\triangleleft_{A})$ is called
{\bf invariant} if \meqref{eq:cor4} holds. A {\bf quadratic \sdpp}
$(A,\triangleright_{A},\triangleleft_{A},\mathcal{B})$ is a \sdpp $(A,\triangleright_{A},\triangleleft_{A})$ together with a nondegenerate symmetric invariant
bilinear form $\mathcal{B}$.
\end{defi}

\begin{pro}\mlabel{pro:330}
Let $\mathcal{B}$ be a nondegenerate symmetric left-invariant
bilinear form on a perm algebra $(A,\circ_{A})$. Define
multiplications $\triangleleft_{A},\triangleright_{A}:A\otimes A\rightarrow A$
respectively by \meqref{eq:cor4} and the following equation
\begin{equation}\mlabel{eq:succ}
x\triangleright_{A}y=x\circ_{A}y-x\triangleleft_{A}y,\;\forall x,y\in A.
\end{equation}
Then $(A,\triangleright_{A},\triangleleft_{A},\mathcal{B})$ is a quadratic
\sdpp.

Conversely, let $(A,\triangleright_{A},\triangleleft_{A},\mathcal{B})$ be
a quadratic \sdpp. Then $\mathcal B$ is left-invariant on the
associated perm algebra $(A,\circ_A)$. Hence there is a
one-to-one correspondence between perm algebras with nondegenerate
symmetric left-invariant bilinear forms and quadratic \sdpps.
\end{pro}

\begin{proof} The first part follows directly from
Corollary \mref{cor:quadratic1}.
The second part follows from Lemma \mref{lem:337}.
\end{proof}

\begin{cor}
\mlabel{cor:2.14} Let $(A,\cdot_{A},\mathcal{B})$ be a symmetric
Frobenius commutative algebra and $P:A\rightarrow A$ be an
averaging operator. Let $(A,\circ_A)$ be the perm algebra
defined by \meqref{eq:perm from aver op}. Then there is a
compatible  \sdpp structure $(A,\triangleright_{A},\triangleleft_{A})$ on
$(A,\circ_A)$ with $\triangleright_A,\triangleleft_A$ defined by
\begin{equation}\mlabel{eq:90}
x\triangleright_{A}y=P(x)\cdot_{A}y+\widehat{P}(x\cdot_{A}y),\; x\triangleleft_{A}y=-\widehat{P}(x\cdot_{A}y),\;\forall x,y\in A.
\end{equation}
Moreover, $(A,\triangleright_{A},\triangleleft_{A},\mathcal{B})$ is quadratic.
\end{cor}

\begin{proof} By Proposition~\mref{pro:2.1}, $\mathcal B$ is a
nondegenerate symmetric left-invariant bilinear form on
$(A,\circ_A)$. By Corollary \mref{cor:quadratic1}, there is a
compatible  \sdpp $(A,\triangleright_{A},\triangleleft_{A},\mathcal{B})$, where
\begin{eqnarray*}
&&\mathcal{B}(x\triangleleft_{A}y,z)=-\mathcal{B}(z\circ_{A}x,y)
=-\mathcal{B}\big(P(z)\cdot_{A}x,y\big)=-\mathcal{B}\big(\widehat{P}(x\cdot_{A}y),z\big),\;\forall x,y,z\in A.\mlabel{eq:hatP1}
\end{eqnarray*}
By the nondegeneracy of $\mathcal{B}$, we have
\begin{equation*}
x\triangleleft_{A}y=-\widehat{P}(x\cdot y), \quad
x\triangleright_{A}y=x\circ_{A}y-x\triangleleft_{A}y=P(x)\cdot_{A}y+\widehat{P}(x\cdot_{A}y),\;\;\forall
x,y\in A.
\end{equation*}
Then it is obvious that $(A,\triangleright_{A},\triangleleft_{A},\mathcal{B})$ is
quadratic.
\end{proof}

\begin{rmk}
The fact that $(A,\triangleright_{A},\triangleleft_{A})$ is a \sdpp in Corollary
\mref{cor:2.14} can also be obtained from \meqref{eq:com asso and
SDPP} in Proposition \mref{pro:com asso and SDPP} since
$(A,\cdot_A,P,\widehat P)$ is an admissible averaging commutative
algebra due to Proposition~\mref{pro:2.6}.
\end{rmk}

\begin{ex}\mlabel{ex:2.21}
Continuing with Example \mref{pre-ex:2.2} for $V=A^{*}$,
$\mu=\mathcal{L}^{*}_{\cdot_{A}}$, there is a quadratic \sdpp
$(A\oplus
A^{*},\triangleright_{d},\triangleleft_{d},\mathcal{B}_{d})$ with
$\triangleright_d$ and $\triangleleft_d$ defined respectively by
\begin{eqnarray*}\mlabel{eq:2.21}
(x+a^{*})\triangleleft_{d}(y+b^{*})&=&-\widehat{P}\big((x+a^{*})\cdot_{d}(y+b^{*})\big)=-\mathcal{L}^{*}_{\cdot_{A}}(x)b^{*}-\mathcal{L}^{*}_{\cdot_{A}}(y)a^{*},\\
(x+a^{*})\triangleright_{d}(y+b^{*})&=&P(x+a^{*})\cdot_{d}(y+b^{*})+\widehat{P}\big((x+a^{*})\cdot_{d}(y+b^{*})\big)\\
    &=&x\cdot_{A}y+2\mathcal{L}^{*}_{\cdot_{A}}(x)b^{*}+\mathcal{L}^{*}_{\cdot_{A}}(y)a^{*},\;\;\forall x,y\in A, a^*,b^*\in A^*.
\end{eqnarray*}
\end{ex}

\begin{ex}
Continuing with Example \mref{pre-ex:2.3}, there is a quadratic
\sdpp $(A\otimes A,\triangleright,\triangleleft,\mathcal{B}')$
with $\triangleright$ and $\triangleleft$ defined respectively by
\begin{eqnarray*}
(x\otimes y)\triangleleft(z\otimes w)&=&-\widehat{P}\big( (x\otimes y)\cdot (z\otimes w)\big)=-x\cdot_{A} z\otimes y\cdot_{A} w-y\cdot_{A} w\otimes x\cdot_{A} z,\\
(x\otimes y)\triangleright(z\otimes w)&=&P(x\otimes y)\cdot(z\otimes w)+\widehat{P}\big( (x\otimes y)\cdot (z\otimes w)\big)\\
&=&y\cdot_{A} z\otimes x\cdot_{A} w+2x\cdot_{A} z\otimes y\cdot_{A}
w+y\cdot_{A} w\otimes x\cdot_{A}
z,\;\;\forall x,y,z,w\in A.
\end{eqnarray*}
\end{ex}

\begin{ex}
Continuing with Example \mref{ex:multiplication}, there
exists a \sdpp $(A,\triangleright_{A},\triangleleft_{A})$ with $\triangleright_A$ and $\triangleleft_A$
defined by \meqref{eq:90}, whose  nonzero products are
\begin{eqnarray*}
&&e_{1}\triangleright_{A} e_{1}=e_{1},\; e_{1}\triangleright_{A} e_{2}=e_{2}\triangleright_{A} e_{1}=e_{2},\;e_{1}\triangleright_{A} e_{3}=e_{2}\triangleright_{A} e_{4}=2e_{3},\;e_{1}\triangleright_{A} e_{4}=2e_{4},\\
&&e_{3}\triangleright_{A} e_{1}=e_{4}\triangleright_{A} e_{2}=e_{3},\;e_{4}\triangleright_{A} e_{1}=e_{4},\;e_{1}\triangleleft_{A} e_{3}=e_{2}\triangleleft_{A} e_{4}=-e_{3},\;e_{1}\triangleleft_{A} e_{4}=-e_{4}.
\end{eqnarray*}
\end{ex}

In the following, we show that \sdpps give rise to both pre-Lie algebras and anti-pre-Lie algebras.

\begin{defi}
    \begin{enumerate}
        \item \mcite{Bai2021.2,Bur}
        A {\bf pre-Lie algebra} is a vector space $A$ together with a multiplication $\star_{A}:A\otimes A\rightarrow A$ such that
        \begin{equation*}
            (x\star_{A} y)\star_{A} z-x\star_{A}(y\star_{A} z)=(y\star_{A} x)\star_{A} z-y\star_{A}(x\star_{A} z),\;\forall x,y,z\in A.
        \end{equation*}
        \item
        \mcite{LB2022}
        An {\bf anti-pre-Lie algebra} is a vector space $A$ together with a multiplication $\diamond_{A}:A\otimes A\rightarrow A$ such that $(A,[-,-]_A)$ is a Lie algebra with the multiplication $[-,-]_A$ defined by
        \vsc
        \begin{eqnarray}\label{eq:sub Lie}
            [x,y]_{A}=x\diamond_{A}y-y\diamond_{A}x,\;\forall x,y\in A,
        \end{eqnarray}
and the following equation holds:
        \begin{eqnarray}
x\diamond_{A}(y\diamond_{A}z)-y\diamond_{A}(x\diamond_{A}z)=[y,x]_{A}\diamond_{A}z,\;\;\forall
            x,y,z\in A.\mlabel{eq:anti-pre-Lie1}
        \end{eqnarray}
    \end{enumerate}
\end{defi}

\begin{lem}\mcite{LB2022}
    Let $\mathcal{B}$ be a nondegenerate commutative $2$-cocycle on a Lie algebra $(A,[-,-]_{A})$ and $\diamond_{A}:A\otimes A\rightarrow A$ be a multiplication on $A$ given by
    \begin{eqnarray}
\mathcal{B}(x\diamond_{A}y,z)=\mathcal{B}(y,[x,z]_{A}),\;\forall x,y,z\in A.
    \end{eqnarray}
Then $(A,\diamond_{A})$ is an anti-pre-Lie algebra which is compatible with $(A,[-,-]_{A})$ in the sense of \eqref{eq:sub Lie}.
\end{lem}

\begin{pro}\mlabel{pro:SDPP and pre-Lie}
    \mlabel{pro:SDPP and anti}
    Let $(A,\triangleright_{A},\triangleleft_{A})$ be a \sdpp. Then $(A,\triangleright_{A})$ is a pre-Lie algebra; while $(A,\diamond_{A}=\triangleright_{A}+2\triangleleft_{A})$ is
    an anti-pre-Lie algebra.
    In particular, if $(A,\triangleright_{A},\triangleleft_{A},\mathcal{B})$ is a quadratic \sdpp, then $(A,\diamond_{A})$ is exactly the anti-pre-Lie algebra induced from the nondegenerate commutative $2$-cocycle $\mathcal{B}$ of the sub-adjacent Lie algebra $(A,[-,-]_{A})$ of the associated perm algebra $(A,\circ_{A})$.
\end{pro}

\begin{proof} Let $x,y,z\in A$.
    By the assumption, we have
    \begin{eqnarray*}
        &&(x\triangleright_{A}y)\triangleright_{A}z-(y\triangleright_{A}x)\triangleright_{A}z-x\triangleright_{A}(y\triangleright_{A}z)
        +y\triangleright_{A}(x\triangleright_{A}z)\\
        &&\overset{\meqref{eq:dpp1.4}}{=}(y\triangleright_{A} z)\triangleleft_{A}x-(x\triangleleft_{A}y)\triangleright_{A}z
        -(x\triangleright_{A} z)\triangleleft_{A}y+(y\triangleleft_{A}x)\triangleright_{A}z\\
        &&=(y\triangleright_{A} z)\triangleleft_{A}x
        -(x\triangleright_{A} z)\triangleleft_{A}y \\
        &&\overset{\meqref{eq:dpp1.5}}{=}-(x\triangleleft_{A}z)\triangleleft_{A}y-x\triangleleft_{A}(y\triangleleft_{A}z)
        +y\triangleleft_{A}(x\triangleleft_{A}z)+(y\triangleleft_{A}z)\triangleleft_{A}x\\
        &&=0.
    \end{eqnarray*}
  Hence $(A,\triangleright_{A})$ is a pre-Lie algebra.

    On the other hand, we also have
        \begin{eqnarray*}
            &&x\diamond_{A}(y\diamond_{A}z)-y\diamond_{A}(x\diamond_{A}z)-[y,x]_{A}\diamond_{A}z\\
            &&=x\triangleright_{A}(y\triangleright_{A}z)+2x\triangleright_{A}(y\triangleleft_{A}z)+2x\triangleleft_{A}(y\triangleright_{A}z)
            +4x\triangleleft_{A}(y\triangleleft_{A}z)-y\triangleright_{A}(x\triangleright_{A}z)\\
            &&\ \
            -2y\triangleright_{A}(x\triangleleft_{A}z)-2y\triangleleft_{A}(x\triangleright_{A}z)-4y\triangleleft_{A}(x\triangleleft_{A}z)
            -(y\triangleright_{A}x)\triangleright_{A}z-2(y\triangleright_{A}x)\triangleleft_{A}z\\
            &&\ \
            +(x\triangleright_{A}y)\triangleright_{A}z+2(x\triangleright_{A}y)\triangleleft_{A}z\\
            &&\overset{\meqref{eq:dpp1.1}}{=}x\triangleright_{A}(y\triangleright_{A}z)+2x\triangleleft_{A}(y\triangleright_{A}z)
            -y\triangleright_{A}(x\triangleright_{A}z)-2y\triangleleft_{A}(x\triangleright_{A}z)-(y\triangleright_{A}x)\triangleright_{A}z\\
            &&\ \
            -2(y\triangleright_{A}x)\triangleleft_{A}z+(x\triangleright_{A}y)\triangleright_{A}z+2(x\triangleright_{A}y)\triangleleft_{A}z\\
            &&\overset{\meqref{eq:dpp1.5}}{=}x\triangleright_{A}(y\triangleright_{A}z)-2x\triangleleft_{A}(y\triangleleft_{A}z)
            -2(x\triangleleft_{A}z)\triangleleft_{A}y-y\triangleright_{A}(x\triangleright_{A}z)+2y\triangleleft_{A}(x\triangleleft_{A}z)\\
            &&\ \
            +2(y\triangleleft_{A}z)\triangleleft_{A}x-(y\triangleright_{A}x)\triangleright_{A}z+2z\triangleleft_{A}(y\triangleleft_{A}x)
            +2(z\triangleleft_{A}x)\triangleleft_{A}y+(x\triangleright_{A}y)\triangleright_{A}z\\
            &&\ \
            -2z\triangleleft_{A}(x\triangleleft_{A}y)-2(z\triangleleft_{A}y)\triangleleft_{A}x\\
            &&=x\triangleright_{A}(y\triangleright_{A}z)-y\triangleright_{A}(x\triangleright_{A}z)
            +2(z\triangleleft_{A}x)\triangleleft_{A}y+(x\triangleright_{A}y-y\triangleright_{A}x)\triangleright_{A}z
            -2(z\triangleleft_{A}y)\triangleleft_{A}x\\
            &&\overset{\meqref{eq:dpp1.4}}{=}y\triangleright_{A}(x\triangleleft_{A}z)-x\triangleright_{A}(y\triangleleft_{A}z)
            +2(z\triangleleft_{A}x)\triangleleft_{A}y-2(z\triangleleft_{A}y)\triangleleft_{A}x\\
            &&\overset{\meqref{eq:dpp1.1}}{=}0.
        \end{eqnarray*}
    Since $\triangleleft_{A}$ is commutative, we have
    \begin{eqnarray*}
&&  [x,y]_A:=x\diamond_A y-y\diamond_A x=
    x\triangleright_{A}y+2x\triangleleft_{A}y-y\triangleright_{A}x-2y\triangleleft_{A}x\\
    &&= x\triangleright_{A}y+ x\triangleleft_{A}y-y\triangleright_{A}x- y\triangleleft_{A}x
    =x\circ_{A}y-y\circ_{A}x,\;\;\forall x,y\in A.
    \end{eqnarray*}
    Since $(A,\circ_{A})$ is a perm algebra,
    $(A,[-,-]_{A})$ is a Lie algebra.
     Hence $(A,\diamond_{A})$
    is an anti-pre-Lie algebra.

In particular, if $(A,\triangleright_{A},\triangleleft_{A},\mathcal{B})$ is a quadratic \sdpp, then we have
\begin{eqnarray*}
\mathcal{B}(x\diamond_{A}y,z)=\mathcal{B}(x\circ_{A}y,z)+\mathcal{B}(x\triangleleft_{A} y,z)\overset{\meqref{eq:cor4}}{=}\mathcal{B}(y,x\circ_{A}z)-\mathcal{B}(y,z\circ_{A}x)=\mathcal{B}(y,[x,z]_{A}).
\end{eqnarray*}
Hence   $(A,\diamond_{A})$ is exactly the anti-pre-Lie algebra induced from the nondegenerate commutative $2$-cocycle $\mathcal{B}$ of the sub-adjacent Lie algebra $(A,[-,-]_{A})$ of the associated perm algebra $(A,\circ_{A})$.
\end{proof}

\section{Special apre-perm bialgebras}\mlabel{sec:5}\
We introduce the notions of Manin triples of \sdpps and Manin
triples of perm algebras associated to nondegenerate symmetric
left-invariant bilinear forms, together with their relationship.
Then we introduce the notion of a \sdppb as an equivalent
structure of a Manin triple of \sdpps, which can be derived from
an averaging commutative and cocommutative infinitesimal bialgebra.

\subsection{Manin triples of perm algebras and  \sdpps}\mlabel{sec3.1}\

\begin{defi}
    \begin{enumerate}
        \item Let $(A,\circ_{A})$ and
        $(A^{*},\circ_{A^{*}})$ be perm algebras. If there is a perm algebra
        structure $(A\oplus
        A^{*},\circ_{d})$ on $A\oplus A^{*}$
        containing $(A,\circ_{A})$ and $(
        A^{*},\circ_{A^{*}})$ as perm subalgebras, and the bilinear form $\mathcal{B}_{d}$ in \meqref{eq:bfds} is left-invariant on $(A\oplus
        A^{*},\circ_{d})$, then we say that
        $\big(  (  A\oplus
        A^{*},\circ_{d},\mathcal{B}_{d}),(A,\circ_{A}),(A^{*},\circ_{A^{*}})\big)
        $ is a {\bf Manin triple of perm algebras associated to the nondegenerate
            symmetric left-invariant bilinear form $\mathcal{B}_d$}.
        \item Let $(A,\triangleright_{A},\triangleleft_{A})$ and
        $(A^{*},\triangleright_{A^{*}},\triangleleft_{A^{*}})$ be \sdpps. If there is a
        quadratic \sdpp structure $(A\oplus
        A^{*},\triangleright_{d},\triangleleft_{d},\mathcal{B}_{d})$ on $A\oplus A^{*}$
        which contains $(A,\triangleright_{A},\triangleleft_{A})$ and $(
        A^{*},\triangleright_{A^{*}},\triangleleft_{A^{*}})$ as \sdppsubs, then we say
        $\big(  (  A\oplus
        A^{*},\triangleright_{d},\triangleleft_{d},\mathcal{B}_{d}),
        (A,\triangleright_{A},\triangleleft_{A}),$
        $(A^{*},\triangleright_{A^{*}},\triangleleft_{A^{*}})\big)
        $ is a \textbf{Manin triple of \sdpps} (associated to the nondegenerate symmetric invariant bilinear form $\mathcal{B}_d$).
    \end{enumerate}
\end{defi}

\begin{pro}\mlabel{pro:equ}
There is a one-to-one correspondence between
Manin triples of perm algebras associated to nondegenerate
symmetric left-invariant bilinear forms and Manin triples of
\sdpps.
\end{pro}
\begin{proof}
Let $\big(  (  A\oplus
A^{*},\circ_{d},\mathcal{B}_{d}),(A,\circ_{A}),(A^{*},\circ_{A^{*}})\big)
$ be a Manin triple of perm algebras associated to the nondegenerate
    symmetric left-invariant bilinear form $\mathcal{B}_d$. Then  by Proposition \mref{pro:330}, there is a quadratic \sdpp  $(A\oplus A^{*},\triangleright_{d},\triangleleft_{d},\mathcal{B}_{d})$ with the multiplications
    $\triangleright_{d}$ and $\triangleleft_{d}$ defined respectively by
    \begin{eqnarray*}
    &&\mathcal{B}_{d}\big( (x+a^{*})\triangleleft_{d}(y+b^{*}) ,z+c^{*}\big)=-\mathcal{B}_{d}\big( x+a^{*}, (z+c^{*})\circ_{d}(y+b^{*}) \big),\\
    &&(x+a^{*})\triangleright_{d}(y+b^{*})=(x+a^{*})\circ_{d}(y+b^{*})-(x+a^{*})\triangleleft_{d}(y+b^{*}),\;\forall x,y,z\in A, a^{*},b^{*},c^{*}\in A^{*}.
    \end{eqnarray*}In particular, we have
\begin{eqnarray*}
    \mathcal{B}_{d}(x\triangleleft_{d}y,z)=-\mathcal{B}_{d}(x,z\circ_{d}y)=0,\;\;\forall
    x,y,z\in A.
\end{eqnarray*}
Hence we have $x\triangleleft_{d}y\in A$, and $x\triangleright_{d}y=x\circ_{d}y-x\triangleleft_{d}y=x\circ_{A}y-x\triangleleft_{d}y\in A$.
Thus $(A,\triangleright_{A},\triangleleft_{A}):=(A,\triangleright_{d}|_{A},\triangleleft_{d}|_{A})$ is a subalgebra of $(A\oplus A^{*},\triangleright_{d},\triangleleft_{d})$.
Similarly, $(A^{*},\triangleright_{A^{*}},\triangleleft_{A^{*}})$ is also a subalgebra of $(A\oplus A^{*},\triangleright_{d},\triangleleft_{d})$.
Hence $\big(  (  A\oplus
A^{*},\triangleright_{d},\triangleleft_{d},\mathcal{B}_{d}),(A,\triangleright_{A},\triangleleft_{A}),(A^{*},\triangleright_{A^{*}},\triangleleft_{A^{*}})\big)
$ is a  Manin triple of \sdpps.

Conversely, suppose that $\big(  (  A\oplus
A^{*},\triangleright_{d},\triangleleft_{d},\mathcal{B}_{d}),(A,\triangleright_{A},\triangleleft_{A}),(A^{*},\triangleright_{A^{*}},\triangleleft_{A^{*}})\big)
$ is a  Manin triple of \sdpps. Then $\big(  (  A\oplus
A^{*},\circ_{d},\mathcal{B}_{d}),(A,\circ_{A}),(A^{*},\circ_{A^{*}})\big)
$ is straightforwardly a Manin triple of perm algebras associated
to the nondegenerate symmetric left-invariant bilinear form $\mathcal{B}_d$, where $(A\oplus
A^{*},\circ_{d}),(A,\circ_{A})$ and $(A^{*},\circ_{A^{*}})$ are
the associated perm algebras of $(A\oplus
A^{*},\triangleright_{d},\triangleleft_{d})$,
$(A,\triangleright_{A},\triangleleft_{A})$ and
$(A^{*},\triangleright_{A^{*}},\triangleleft_{A^{*}})$
respectively.
\end{proof}

\begin{pro}\mlabel{pro:5.2}
Let $\big( (A\oplus
A^{*},\cdot_{d},P+Q^{*},\mathcal{B}_{d}),(A,\cdot_{A},P),(A^{*},\cdot_{A^{*}},Q^{*})
\big)$ be a double construction of averaging Frobenius commutative
algebra. Then there is a Manin triple of \sdpps $\big(  (
A\oplus
A^{*},\triangleright_{d},\triangleleft_{d},\mathcal{B}_{d}),(A,\triangleright_{A},\triangleleft_{A}),(A^{*},
\triangleright_{A^{*}}, \triangleleft_{A^{*}})\big)$, where
\begin{small}
\begin{eqnarray}
    && (x+a^{*})\triangleright_{d}(y+b^{*}):=(P+Q^{*})(x+a^{*})\cdot_{d}(y+b^{*})+(Q+P^{*})\big( (x+a^{*})\cdot_{d}(y+b^{*})\big),\mlabel{eq:dc1}\\
    &&(x+a^{*})\triangleleft_{d}(y+b^{*}):=-(Q+P^{*})\big( (x+a^{*})\cdot_{d}(y+b^{*})\big),\;\forall x,y\in A, a^{*},b^{*}\in A^{*}.\mlabel{eq:dc2}
\end{eqnarray}
\end{small}
\end{pro}
\begin{proof}
The adjoint map of $P+Q^{*}$ with respect to $\mathcal{B}_{d}$ is
$Q+P^{*}$. Hence by Corollary \mref{cor:2.14}, there is a
quadratic \sdpp $( A\oplus
A^{*},\triangleright_{d},\triangleleft_{d},\mathcal{B}_{d} )$ with
$\triangleright_{d},\triangleleft_{d}$ defined by \meqref{eq:dc1}
and \meqref{eq:dc2} respectively. Moreover, by Theorem
\mref{thm:2.11}, $(A,\cdot_{A},P,Q)$ and
$(A^{*},\cdot_{A^{*}},Q^{*},P^{*})$ are admissible averaging
commutative algebras. Hence by Proposition
\mref{pro:com asso and SDPP}, there are \sdpps
$(A,\triangleright_{A},\triangleleft_{A})$ and
$(A^{*},\triangleright_{A^{*}},\triangleleft_{A^{*}})$, where $
\triangleright_{A}$ and $\triangleleft_{A} $ are defined by
\meqref{eq:com asso and SDPP}, and $ \triangleright_{A^{*}}$ and
$\triangleleft_{A^{*}} $ are defined by
\begin{eqnarray}
    a^{*}\triangleright_{A^{*}}b^{*}:=Q^{*}(a^{*})\cdot_{A^{*}}b^{*}+P^{*}(a^{*}\cdot_{A^{*}}b^{*}),\; a^{*}\triangleleft_{A^{*}}b^{*}:=-P^{*}(a^{*}\cdot_{A^{*}}b^{*}),\;\forall a^{*},b^{*}\in A^{*}, \mlabel{eq:mp re2}
\end{eqnarray}
which serve as \sdppsubs of $(A\oplus A^{*},\triangleright_{d},\triangleleft_{d})$.
Hence $\big((A\oplus A^{*},\triangleright_{d},\triangleleft_{d},\mathcal{B}_{d}),$
$(A,\triangleright_{A},\triangleleft_{A}),(A^{*},\triangleright_{A^{*}},
\triangleleft_{A^{*}})\big) $ 
is a Manin triple of \sdpps.
\end{proof}

\begin{thm}\mlabel{thm:Manin triple}
Let
$(A,\triangleright_{A},\triangleleft_{A})$ and
$(A^{*},\triangleright_{A^{*}},\triangleleft_{A^{*}})$ be \sdpps,
and let their associated perm algebras be $(A,\circ_{A})$ and
$(A^{*},\circ_{A^{*}})$ respectively.
 Then the
following conditions are equivalent:
\begin{enumerate}
\item\mlabel{E4} There is a Manin triple $\big(  (  A\oplus
A^{*},\circ_{d},\mathcal{B}_{d}),(A,\circ_{A}),(A^{*},\circ_{A^{*}})\big)
$ of perm algebras associated to the nondegenerate symmetric
left-invariant bilinear form $\mathcal{B}_d$ such that the
compatible \sdpp $(A\oplus
A^{*},\triangleright_{d},\triangleleft_{d})$ induced from
$\mathcal{B}_{d}$ contains
$(A,\triangleright_{A},\triangleleft_{A})$ and
$(A^{*},\triangleright_{A^{*}},\triangleleft_{A^{*}})$ as
subalgebras. \item\mlabel{E1} There is a Manin triple
$\big((A\oplus
A^{*},\triangleright_{d},\triangleleft_{d},\mathcal{B}_{d}),(A,\triangleright_{A},\triangleleft_{A}),(A^{*},\triangleright_{A^{*}},
\triangleleft_{A^{*}})\big)$ of \sdpps. \item\mlabel{E2} There is
a perm algebra $(A\oplus A^{*},\circ_{d})$ with $ \circ_{d} $
defined by
\begin{equation}\mlabel{eq:A ds}
    (x+a^{*})\circ_{d}(y+b^{*})=x\circ_{A}y+\mathcal{L}^{*}_{\circ_{A^{*}}}(a^{*})y-\mathcal{L}^{*}_{\triangleleft_{A^{*}}}(b^{*})x+a^{*}\circ_{A^{*}}b^{*}+
    \mathcal{L}^{*}_{\circ_{A}}(x)b^{*}-\mathcal{L}^{*}_{\triangleleft_{A}}(y)a^{*},
\end{equation}
for all $x,y\in A, a^{*},b^{*}\in A^{*}$. \item\mlabel{E3} There
is a \sdpp  $(A\oplus A^{*},\triangleright_{d},\triangleleft_{d})$
with $\triangleright_{d}$ and $\triangleleft_{d}$ defined respectively by
\begin{eqnarray}
    (x+a^{*})\triangleright_{d}(y+b^{*})&=&\;x\triangleright_{A}y+(\mathcal{L}^{*}_{\circ_{A^{*}}}+\mathcal{R}^{*}_{\circ_{A^{*}}})(a^{*})y
    +\mathcal{R}^{*}_{\triangleright_{A^{*}}}(b^{*})x\nonumber\\
    &&\;+a^{*}\triangleright_{A^{*}}b^{*}+(\mathcal{L}^{*}_{\circ_{A}}+\mathcal{R}^{*}_{\circ_{A}})(x)b^{*}+\mathcal{R}^{*}_{\triangleright_{A}}(y)a^{*},\mlabel{eq:A ds1}\\
    (x+a^{*})\triangleleft_{d}(y+b^{*})&=&\;x\triangleleft_{A}y-\mathcal{R}^{*}_{\circ_{A^{*}}}(a^{*})y
    -\mathcal{R}^{*}_{\circ_{A^{*}}}(b^{*})x\nonumber\\
    &&\;+a^{*}\triangleleft_{A^{*}}b^{*}-\mathcal{R}^{*}_{\circ_{A }}(x)b^{*}
    -\mathcal{R}^{*}_{\circ_{A }}(y)a^{*},\mlabel{eq:A ds2}
\end{eqnarray}
for all $x,y\in A, a^{*},b^{*}\in A^{*}$.
\end{enumerate}
\end{thm}

\begin{proof}
(\mref{E4})$\Longleftrightarrow$(\mref{E1}). It follows from
Proposition \mref{pro:equ}.

(\mref{E1})$\Longrightarrow$(\mref{E3}). Let $x,y\in A$ and
$ a^{*},b^{*}\in A^{*}$. By the assumption, we have
\begin{eqnarray*}
\mathcal{B}_{d}(x\triangleright_{d} b^{*},y)&\overset{\meqref{eq:cor3}}{=}&\mathcal{B}
_{d}(b^{*},x\circ_{A} y+y\circ_{A} x)=
\mathcal{B}_{d}\big((\mathcal{L}
^{*}_{\circ_{A}}+\mathcal{R}
^{*}_{\circ_{A}})(x)b^{*},y\big),\\
\mathcal{B}_{d}(x\triangleright_{d} b^{*},a^{*})&\overset{\meqref{eq:cor3},\meqref{eq:cor4}}{=}&
\mathcal{B}_{d}(x,a^{*}\triangleright_{A^{*}}b^{*})=\langle x,
a^{*}\triangleright_{A^{*}}b^{*}\rangle=\mathcal{B}_{d}\big(\mathcal{R}
^{*}_{\triangleright_{A^{*}}}(b^{*})x,a^{*}\big).
\end{eqnarray*}
Thus we have
\[
x\triangleright_{d} b^{*}=(\mathcal{L}
^{*}_{\circ_{A}}+\mathcal{R}
^{*}_{\circ_{A}})(x)b^{*}+\mathcal{R}
^{*}_{\triangleright_{A^{*}}}(b^{*})x,
\]
and similarly
\[
a^{*}\triangleright_{d} y=(\mathcal{L}
^{*}_{\circ_{A^{*}}}+\mathcal{R}
^{*}_{\circ_{A^{*}}})(a^{*})y+\mathcal{R}
^{*}_{\triangleright_{A}}(y)a^{*}.
\]
Hence \meqref{eq:A ds1} holds. Moreover, we have
\begin{eqnarray*}
\mathcal{B}_{d}(x\triangleleft_{d} b^{*},y)&\overset{\meqref{eq:cor4}}{=}&-\mathcal{B}
_{d}(b^{*},y\circ_{A} x)=-\langle \mathcal{R}^{*}_{\circ_{A}}(x)b^{*},y\rangle=-\mathcal{B}_{d}\big(\mathcal{R}^{*}_{\circ_{A}}(x)b^{*},y\big), \\
\mathcal{B}_{d}(x\triangleleft_{d} b^{*},a^{*})&\overset{\meqref{eq:cor4}}{=}&
-\mathcal{B}_{d}(x,a^{*}\circ_{A^{*}}b^{*})=-\langle \mathcal{R}^{*}_{
    \circ_{A^{*}}} (b^{*})x, a^{*}\rangle=-\mathcal{B}_{d}\big(\mathcal{R}^{*}_{
    \circ_{A^{*}}}(b^{*})x, a^{*} \big).
\end{eqnarray*}
Thus we have
\[
x\triangleleft_{d} b^{*}=-\mathcal{R}
^{*}_{\circ_{A}} (x)b^{*}-\mathcal{R}^{*}_{
\circ_{A^{*}}} (b^{*})x.
\]
Hence \meqref{eq:A ds2} holds.

(\mref{E3})$\Longrightarrow$(\mref{E2}). It is straightforward.

(\mref{E2})$\Longrightarrow$(\mref{E1}). Suppose that there is a
perm algebra $(A\oplus A^{*},\circ_{d})$ with $\circ_{d}$ defined
by \meqref{eq:A ds}. Then it is straightforward to check that
$\mathcal{B}_{d}$ is left-invariant on $(A\oplus
A^{*},\circ_{d})$. By Proposition \mref{pro:330}, there is a
quadratic \sdpp $(A\oplus
A^{*},\triangleright_{d},\triangleleft_{d},\mathcal{B}_{d})$ with
$\triangleright_{d},\triangleleft_{d}$ defined by
 \meqref{eq:cor3} and \meqref{eq:cor4} respectively. Moreover, let $x,y,z\in
A$. Then we have
\begin{equation*}
\mathcal{B}_{d}(x\triangleleft_{d}y,z)=-\mathcal{B}_{d}(x,z\circ_{d}y)=-\mathcal{B}_{d}(x,z\circ_{A}y)=0,
\end{equation*}
which indicates $x\triangleleft_{d}y\in A$. Similarly we have
$$x\triangleright_{d}y\in A,\;a^{*}\triangleright_{d}b^{*}\in A^{*},\;a^{*}\triangleleft_{d}b^{*}\in A^{*},\;\forall x,y\in A, a^{*},b^{*}\in A^{*}.$$
Hence $(A,\triangleright_{A}=\triangleright_{d}|_{A})$ and
$(A^{*},\triangleright_{A^{*}}=\triangleright_{d}|_{A^{*}})$ are subalgebras of
$(A\oplus A^{*},\triangleright_{d},\triangleleft_{d})$. Therefore $\big((A\oplus
A^{*},\triangleright_{d},\triangleleft_{d},\mathcal{B}_{d}),(A,\triangleright_{A},\triangleleft_{A}),
(A^{*},\triangleright_{A^{*}},\triangleleft_{A^{*}})\big) $ is a Manin triple of
\sdpps.
\end{proof}

Recall the notion of Manin triples of Lie algebras associated to commutative $2$-cocycles.

\begin{defi}\mcite{LB2024}
Let $(A,[-,-]_{A})$ and $(A^{*},[-,-]_{A^{*}})$ be Lie algebras.
If there is a Lie algebra structure $(A\oplus A^{*},[-,-]_{d})$ on
$A\oplus A^{*}$ containing $(A,[-,-]_{A})$ and
$(A^{*},[-,-]_{A^{*}})$ as Lie subalgebras, and the bilinear form $\mathcal{B}_{d}$ in \meqref{eq:bfds} is
a commutative $2$-cocycle on $(A\oplus A^{*},[-,-]_{d})$, then we
call $\big((A\oplus
A^{*},[-,-]_{d},\mathcal{B}_{d}),(A,[-,-]_{A}),(A^{*},[-,-]_{A^{*}})\big)$
a {\bf Manin triple of Lie algebras associated to the commutative
$2$-cocycle $\mathcal{B}_d$}.
\end{defi}

\begin{pro}\mlabel{pro:antiManin}
Let $\big(  (  A\oplus
A^{*},\circ_{d},\mathcal{B}_{d}),(A,\circ_{A}),(A^{*},\circ_{A^{*}})\big)
$ be a Manin triple of perm algebras associated to the nondegenerate
    symmetric left-invariant bilinear form $\mathcal{B}_d$, and let the sub-adjacent Lie algebras of
$(A\oplus A^{*},\circ_{d})$, $(A,\circ_{A})$ and
$(A^{*},\circ_{A^{*}})$ be $ (A\oplus A^{*},[-,-]_{d}
),(A,[-,-]_{A})$ and $(A^{*},[-,-]_{A^{*}})$ respectively. Then
$\big((A\oplus A^{*},[-,-]_{d},\mathcal{B}_{d})$, $(A,[-,-]_{A})$,
$(A^{*},[-,-]_{A^{*}})\big)$ is a Manin triple of Lie algebras
associated to the commutative $2$-cocycle $\mathcal{B}_d$.
\end{pro}
\begin{proof}
It is straightforward to check that $(A,[-,-]_{A} )$ and $(A^{*},[-,-]_{A^{*}})$ are Lie subalgebras of $ (A\oplus A^{*},[-,-]_{d} )$.
By Propositions  \mref{pro:4.6}, $\mathcal{B}_{d}$ is a commutative $2$-cocycle on  $(A\oplus A^{*},[-,-]_{d})$. Hence the conclusion follows.
\end{proof}

\subsection{Special apre-perm bialgebras}\mlabel{sec3.2}\

Recall that a {\bf perm coalgebra} \mcite{Hou,LZB} is a vector
space $A$ with a co-multiplication $\eta:A\rightarrow A\otimes A$
such that
\begin{eqnarray}
(  \eta\otimes\mathrm{id})  \eta(  x)  &=&(
\mathrm{id}\otimes\eta)  \eta(  x),\mlabel{eq:co2}\\
(\mathrm{id}\otimes\eta)  \eta(  x) &=&(\tau\otimes\mathrm{id})(
\eta\otimes\mathrm{id})  \eta(  x),\;\forall x\in A.\mlabel{eq:co3}
\end{eqnarray}

Now we introduce the notion of a special apre-perm coalgebra.

\begin{defi}
Let $A$ be a vector space with co-multiplications $\vartheta,\theta:A\rightarrow A\otimes A$, and $\eta=\vartheta+\theta$.
If $(A,\eta)$ is a perm coalgebra and the following equations hold:
\begin{eqnarray}
\theta(x)&=&\tau\theta(x),\mlabel{eq:co1} \\
(  \eta\otimes\mathrm{id})  \theta(  x)  &=&(
\mathrm{id}\otimes\theta)  \eta(  x),\mlabel{eq:co4}\\
(\mathrm{id}\otimes\theta)  (\eta+\theta)(x) &=&0,\;\forall x\in A,\mlabel{eq:co5}
\end{eqnarray}
then we say $(A, \vartheta,\theta)$ is a {\bf special apre-perm
coalgebra}.
\end{defi}

\begin{pro}\mlabel{lem:co}
Let $A$ be a vector space and $\vartheta,\theta :A\rightarrow A\otimes A$ be co-multiplications.
Let $\triangleright_{A^{*}},\triangleleft_{A^{*}}:A^{*}\otimes A^{*}\rightarrow A^{*}$ be the linear duals of $\vartheta$ and $\theta$ respectively, that is, the following equations hold:
\begin{equation}
\langle a^{*}\triangleright_{A^{*}}b^{*},x\rangle=\langle a^{*}\otimes b^{*},\vartheta(x)\rangle,\;\langle a^{*}\triangleleft_{A^{*}}b^{*},x\rangle=\langle a^{*}\otimes b^{*},\theta(x)\rangle,\;\forall x\in A, a^{*},b^{*}\in A^{*}.
\end{equation}
Then $(A^{*},\triangleright_{A^{*}},\triangleleft_{A^{*}})$ is a
\sdpp if and only if $(A,\vartheta,\theta )$ is a special
apre-perm coalgebra.
\end{pro}
\begin{proof}
Let $\eta=\vartheta+\theta$ and the linear dual of $\eta$ be $\circ_{A^{*}}$.
From a straightforward computation, we have
\begin{enumerate}
\item $(A,\eta)$ is a perm coalgebra if and only if $(A^{*},\circ_{A^{*}})$ is a perm algebra.
\item \meqref{eq:co1} holds if and only if $\triangleleft_{A^{*}}$ is commutative.
\item \meqref{eq:co4} and \meqref{eq:co5} hold if and only if \meqref{eq:SDPP} holds on $A^{*}$.
\end{enumerate}
Hence the conclusion follows from Corollary \mref{cor:SDPP}.
\end{proof}

Now we give the notion of a \sdppb.

\begin{defi}
Let $(A,\triangleright_{A},\triangleleft_{A})$ be a \sdpp and let
$(A,\vartheta,\theta)$ be a special apre-perm coalgebra satisfying
the following equations
\begin{eqnarray}
\eta(  x\circ_{A}y)  &=&\big(  \mathcal{L}_{\circ_{A}}(
x)  \otimes\mathrm{id}\big)  \eta(  y)  -\big(
\mathrm{id}\otimes\mathcal{R}_{\circ_{A}}(  y)  \big)
\theta(  x),  \mlabel{eq:bialg1}\\
\eta(  x\circ_{A}y)  &=&
\big(  \mathrm{id}\otimes\mathcal{R}_{\circ_{A}}(  y)  \big)  \eta(  x)-\big(
\mathcal{L}_{\triangleleft_{A}}(  x)  \otimes\mathrm{id}\big)
\eta(  y) ,\mlabel{eq:bialg2}\\
\eta(  x\circ_{A}y)  &=&\big(  \mathrm{id}\otimes\mathcal{L}%
_{\circ_{A}}(  x)  \big)  \eta(  y)  +\big(
\mathcal{L}_{\triangleleft_{A}}(  y)  \otimes\mathrm{id}\big)
\theta(  x),\mlabel{eq:bialg3}\\
\eta(  x\triangleleft_{A}y)&=&\big(  \mathrm{id}\otimes\mathcal{L}_{\triangleleft_{A}}(  x)  \big)
\eta(  y)  +\tau\big(  \mathrm{id}\otimes\mathcal{L}_{\triangleleft_{A}
}(  y)  \big)  \eta(  x), \mlabel{eq:bialg4}\\
\eta(  x\triangleleft_{A}y)  &=&\tau\eta(  x\triangleleft_{A}y),\mlabel{eq:bialg5}\\
\theta(  x\circ_{A}y)&=&\big(
\mathrm{id}\otimes\mathcal{L}_{\circ_{A}}(  x)  \big)
\theta(  y)  +\big(  \mathcal{L}_{\circ_{A}}(  y)
\otimes\mathrm{id}\big)  \theta(  x),\mlabel{eq:bialg6}\\
\theta(  x\circ_{A}y)  &=&\theta(  y\circ_{A}x), \quad \forall x,
y\in A, \mlabel{eq:bialg7}
\end{eqnarray}
where $\eta=\vartheta+\theta$. Then we say that
$(A,\triangleright_{A},\triangleleft_{A},\vartheta,\theta)$ is a
\textbf{\sdppb.}
\end{defi}

\begin{lem}\mcite{Hou,LZB}\mlabel{lem:mp}
Let $(A,\circ_{A})$ and $(B,\circ_{B})$ be two perm algebras, and
$l_{A},r_{A}:A\rightarrow \mathrm{End}_{\mathbb K}(B)$ and
$l_{B},r_{B}:B\rightarrow\mathrm{End}_{\mathbb K}(A)$ be linear
maps. Then there is a perm algebra structure $(A\oplus B,\circ_{d})$ on $A\oplus B$ with $\circ_{d}$ defined
by
\begin{equation}
(x+a)\circ_{d}(y+b)=x\circ_{A}y+l_{B}(a)y+r_{B}(b)x+a\circ_{B}b+l_{A}(x)b+r_{A}(y)a,\;\forall x,y\in A, a,b\in B
\end{equation}
if and only if $(l_{A},r_{A},B)$ and $(l_{B},r_{B},A)$ are representations of $(A,\circ_{A})$ and $(B,\circ_{B})$ respectively, and satisfy the following equations$:$
\begin{eqnarray}
l_{_{A}}(  x)  (  a\circ_{B}b)  &=&l_{_{A}}(
x)  a\circ_{B}b+l_{_{A}}(  r_{_{B}}(  a)  x)  b,\mlabel{eq:mp1}\\
l_{_{A}}(  x)  (  a\circ_{B}b)&=&r_{_{A}}(  x)  a\circ_{B}b+l_{_{A}}(
l_{_{B}}(  a)  x)  b,\mlabel{eq:mp2}\\
l_{_{A}}(  x)  (  a\circ_{B}b)&=&a\circ_{B}l_{_{A}}(  x)  b+r_{_{A}}(
r_{_{B}}(  b)  x)  a,\mlabel{eq:mp3}\\
r_{_{A}}(  x)  (  a\circ_{B}b)&=&a\circ_{B}r_{_{A}%
}(  x)  b+r_{_{A}}(  l_{_{B}}(  b)  x)  a,\mlabel{eq:mp4}\\
r_{_{A}}(  x)  (  a\circ_{B}b)  &=&r_{_{A}}(
x)  (  b\circ_{B}a),\mlabel{eq:mp5}\\
l_{_{B}}(  a)  (  x\circ_{A}y)  &=&l_{_{B}}(
a)  x\circ_{A}y+l_{_{B}}(  r_{_{A}}(  x)  a)  y,
\mlabel{eq:mp6}\\
l_{_{B}}(  a)  (  x\circ_{A}y)&=&r_{_{B}}(  a)  x\circ_{A}y+l_{_{B}}(
l_{_{A}}(  x)  a)  y,\mlabel{eq:mp7}\\
l_{_{B}}(  a)  (  x\circ_{A}y)&=&x\circ_{A}l_{_{B}}(  a)  y+r_{_{B}}(
r_{_{A}}(  y)  a)  x,\mlabel{eq:mp8}\\
r_{_{B}}(  a)  (  x\circ_{A}y)  &=&x\circ_{A}r_{_{B}%
}(  a)  y+r_{_{B}}(  l_{_{A}}(  y)  a)  x,\mlabel{eq:mp9}\\
r_{_{B}}(  a)  (  x\circ_{A}y)  &=&r_{_{B}}(
a)  (  y\circ_{A}x).\mlabel{eq:mp10}
\end{eqnarray}
\end{lem}

\begin{thm}\mlabel{thm:2-2}
Let $(A,\triangleright_{A},\triangleleft_{A})$ and
$(A^{*},\triangleright_{A^{*}},\triangleleft_{A^{*}})$ be \sdpps, and let the linear
maps $\vartheta,\theta:A\rightarrow A\otimes A$ be the linear
duals of $\triangleright_{A^{*}}$ and $\triangleleft_{A^{*}}$ respectively. Then
$(A,\triangleright_{A},\triangleleft_{A},\vartheta,\theta)$ is a \sdppb if and only if there is a perm algebra $(   A\oplus A^{*},\circ_{d})$ with $\circ_{d}$
defined by \meqref{eq:A ds}.
\end{thm}

\begin{proof}
Since $(A,\triangleright_{A},\triangleleft_{A})$ and
$(A^{*},\triangleright_{A^{*}},\triangleleft_{A^{*}})$ are \sdpps,
it follows that $(\mathcal{L}^{*}_{\circ_{A}},-\mathcal{L}
^{*}_{\triangleleft_{A}},$ $A^{*})$ and
$(\mathcal{L}^{*}_{\circ_{A^{*}}},-\mathcal{L}^{*}_{\triangleleft_{A^{*}}},A)$
are representations of the associated perm algebras
$(A,\circ_{A})$ and $ (A^{*},\circ_{A^{*}})$ respectively, and
\meqref{eq:co1}-\meqref{eq:co5} hold by Proposition \mref{lem:co}.
Let $\eta=\vartheta+\theta$ which is the linear dual of
$\circ_{A^{*}}$. For all $x,y\in A, a^{*},b^{*}\in A^{*}$, we have
\begin{eqnarray*}
&&\langle \mathcal{L}_{\circ_{A}}^{*}( x) (
a^{*}\circ_{A^{*}}b^{*}),y\rangle=\langle a^{*}\circ_{A^{*}}b^{*},
x\circ_{A}y\rangle =\langle a^{*}\otimes b^{*},\eta( x\circ_{A}y) \rangle,
\\
&&\langle \mathcal{L}^{*}_{\circ_{A}}(x)a^{*}\circ_{A^{*}}b^{*},y\rangle=
\langle\mathcal{L}^{*}_{\circ_{A}}(x)a^{*}\otimes
b^{*},\eta(y)\rangle=\langle a^{*}\otimes b^{*},\big( \mathcal{L}
_{\circ_{A}}(x)\otimes\mathrm{id} \big)\eta(y)\rangle, \\
&&\langle\mathcal{L}^{*}_{\circ_{A}}\big( -\mathcal{L}^{*}_{
   \triangleleft_{A^{*}}}(a^{*})x\big)b^{*},y\rangle=\langle b^{*},-\mathcal{L}
^{*}_{\triangleleft_{A^{*}}}(a^{*})x\circ_{A}y\rangle=\langle \mathcal{R}
^{*}_{\circ_{A}}(y)b^{*},-\mathcal{L}^{*}_{\triangleleft_{A^{*}}}(a^{*})x\rangle \\
&&=-\langle a^{*}\triangleleft_{A^{*}}\mathcal{R}^{*}_{\circ_{A}}(y)b^{*},x\rangle=
-\langle a^{*}\otimes\mathcal{R}^{*}_{\circ_{A}}(y)b^{*},\theta(x)\rangle=
-\langle a^{*}\otimes b^{*},\big( \mathrm{id}\otimes\mathcal{R}
_{\circ_{A}}(y) \big)\theta(x)\rangle.
\end{eqnarray*}
Thus \meqref{eq:bialg1} holds if and only if \meqref{eq:mp1} holds for $l_{A}=
\mathcal{L}^{*}_{\circ_{A}},r_{A}=-\mathcal{L}^{*}_{\triangleleft_{A}}, l_{B}=
\mathcal{L}^{*}_{\circ_{A^{*}}},r_{B}=-\mathcal{L}^{*}_{\triangleleft_{A^{*}}}. $
Similarly, we have
\begin{eqnarray*}
&&\meqref{eq:bialg1}\Longleftrightarrow\meqref{eq:mp7},\; \meqref{eq:mp2}%
\Longleftrightarrow\meqref{eq:bialg2}\Longleftrightarrow\meqref{eq:mp6},\; %
\meqref{eq:mp3}\Longleftrightarrow\meqref{eq:bialg3}\Longleftrightarrow\meqref{eq:mp8},
\\
&&\meqref{eq:bialg4}\Longleftrightarrow\meqref{eq:mp4},\; \meqref{eq:bialg5}%
\Longleftrightarrow\meqref{eq:mp5},\; \meqref{eq:bialg6}\Longleftrightarrow%
\meqref{eq:mp9},\;
\meqref{eq:bialg7}\Longleftrightarrow\meqref{eq:mp10},
\end{eqnarray*}
where $l_{A}=\mathcal{L}^{*}_{\circ_{A}},r_{A}=-\mathcal{L}^{*}_{\triangleleft_{A}},
l_{B}=\mathcal{L}^{*}_{\circ_{A^{*}}},r_{B}=-\mathcal{L}^{*}_{\triangleleft_{A^{*}}}
$. Hence the conclusion follows from Lemma \mref{lem:mp}.
\end{proof}

Combining Theorems \mref{thm:Manin triple} and \mref{thm:2-2}, we have the following result.

\begin{cor}\mlabel{cor:5.8}
Let $(A,\triangleright_{A},\triangleleft_{A})$ and
$(A^{*},\triangleright_{A^{*}},\triangleleft_{A^{*}})$ be \sdpps.
Then the following conditions are equivalent.
\begin{enumerate}
    \item There is a Manin triple $\big(  (  A\oplus
    A^{*},\circ_{d},\mathcal{B}_{d}),(A,\circ_{A}),(A^{*},\circ_{A^{*}})\big)
    $ of perm algebras associated to the nondegenerate
    symmetric left-invariant bilinear form $\mathcal{B}_d$ such that the compatible \sdpp $(A\oplus A^{*},\triangleright_{d},\triangleleft_{d})$ induced from $\mathcal{B}_{d}$ contains $(A,\triangleright_{A},\triangleleft_{A})$ and
    $(A^{*},\triangleright_{A^{*}},\triangleleft_{A^{*}})$ as subalgebras.
    \item There is a Manin triple $\big(  (  A\oplus
    A^{*},\triangleright_{d},\triangleleft_{d},\mathcal{B}_{d}),(A,\triangleright_{A},\triangleleft_{A}),(A^{*},\triangleright_{A^{*}},\triangleleft_{A^{*}})\big)
    $ of \sdpps.
    \item
    There is a \sdpp  $(A\oplus A^{*},\triangleright_{d},\triangleleft_{d})$ with
$\triangleright_{d},\triangleleft_{d}$ defined by
    \eqref{eq:A ds1} and \eqref{eq:A ds2} respectively.
    \item $(A,\triangleright_{A},\triangleleft_{A},\vartheta,\theta)$ is a \sdppb, where $\vartheta,\theta:A\rightarrow A\otimes A$ are the linear
    duals of $\triangleright_{A^{*}}$ and $\triangleleft_{A^{*}}$ respectively.
\end{enumerate}
\end{cor}

\begin{pro}\mlabel{pro:5.9}
Let $(A,\cdot_{A},\Delta,P,Q)$ be an averaging commutative and cocommutative infinitesimal bialgebra.
Then there is a \sdppb $(A,\triangleright_{A},\triangleleft_{A},\vartheta,\theta)$, where $\triangleright_{A},\triangleleft_{A}$ are defined by \meqref{eq:com asso and SDPP} and $\vartheta,\theta$ are defined by
\begin{equation}\mlabel{eq:co ao}
\vartheta(x)=(Q\otimes\mathrm{id})\Delta(x)+\Delta(Px),\; \theta(x)=-\Delta(Px),\;\forall x\in A.
\end{equation}
\end{pro}
\begin{proof}
By Theorem \mref{thm:2.11} and Proposition \mref{pro:com asso and
SDPP}, $(A,\triangleright_{A},\triangleleft_{A})$ with $\triangleright_{A},\triangleleft_{A}$ defined by \meqref{eq:com asso
and SDPP} and
$(A^{*},\triangleright_{A^{*}},$
$\triangleleft_{A^{*}})$ with $ \triangleright_{A^{*}},\triangleleft_{A^{*}} $ defined by \meqref{eq:mp re2} respectively are \sdpps. The
linear duals $\vartheta,\theta:A\rightarrow A\otimes A$ of
$\triangleright_{A^{*}}$ and $\triangleleft_{A^{*}}$ are given by
\begin{eqnarray*}
    &&\langle \theta(x), a^{*}\otimes b^{*}\rangle=\langle x, a^{*}\triangleleft_{A^{*}}b^{*}\rangle\overset{\meqref{eq:mp re2}}{=}-\langle x, P^{*}(a^{*} \cdot_{A^{*}}b^{*})\rangle=-\langle  \Delta(Px),a^{*}\otimes b^{*}\rangle,\\
    &&\langle \vartheta(x), a^{*}\otimes b^{*}\rangle=\langle x, a^{*}\triangleright_{A^{*}}b^{*}\rangle\overset{\meqref{eq:mp re2}}{=}\langle x, Q^{*}(a^{*})\cdot_{A^{*}}b^{*}+P^{*}(a^{*} \cdot_{A^{*}}b^{*})\rangle\\
    &&=\langle (Q\otimes\mathrm{id})\Delta(x)+\Delta(Px),a^{*}\otimes b^{*}\rangle,\;\;\forall x\in A, a^{*},b^{*}\in A^{*},
\end{eqnarray*}that is, \meqref{eq:co ao} holds. By Proposition \mref{lem:co}, $(A,\vartheta,\theta)$ is a special apre-perm
coalgebra.
Let $\eta=\vartheta+\theta$.
By \meqref{eq:perm
from aver op}, \meqref{eq:com asso and SDPP} and \meqref{eq:mp
re2}, we have
\begin{eqnarray*} \mathcal{L}_{\triangleright_{A}}(y)=\mathcal{L}_{\cdot_{A}}\big(P(y)\big)+Q\mathcal{L}_{\cdot_{A}}(y),\; \mathcal{R}_{\triangleright_{A}}(y)=\mathcal{L}_{\cdot_{A}}(y)P+Q\mathcal{L}_{\cdot_{A}}(y),\;\mathcal{L}_{\triangleleft_{A}}(y)=-Q\mathcal{L}_{\cdot_{A}}(y),
\end{eqnarray*}
for all $y\in A$. Therefore we have
\begin{eqnarray*}
    \eta(x\circ_{A}y)&=&(Q\otimes\mathrm{id})\Delta(x\circ_{A}y)\overset{\meqref{eq:perm from aver op}}{=}(Q\otimes\mathrm{id})\Delta\big(P(x)\cdot_{A}y\big)\\
    &\overset{\meqref{eq:bib}}{=}&(Q\otimes\mathrm{id})\bigg(\Big(\mathcal{L}_{\cdot_{A}}\big(P(x)\big)\otimes\mathrm{id}\Big)\Delta(y)
    +\big(\mathrm{id}\otimes\mathcal{L}_{\cdot_{A}}(y)\big)\Delta\big(P(x)\big)\bigg),\\
    \big(\mathcal{L}_{\circ_{A}} (x)\otimes\mathrm{id}\big)\eta(y)
    &=&\Big(\mathcal{L}_{\cdot_{A}}\big(P(x)\big)\otimes\mathrm{id}\Big)(Q\otimes\mathrm{id})\Delta(y)=(Q\otimes\mathrm{id})\Big(\mathcal{L}_{\cdot_{A}}\big(P(x)\big)\otimes\mathrm{id}\Big)\Delta(y),\\
    -\big(\mathrm{id}\otimes\mathcal{R}_{\circ_{A}}(y)\big)\theta(x)
    &=&\big(\mathrm{id}\otimes\mathcal{L}_{\cdot_{A}}(y)P\big)\Delta\big(P(x)\big)=\big(\mathrm{id}\otimes\mathcal{L}_{\cdot_{A}}(y)\big)(\mathrm{id}\otimes P)\Delta\big(P(x)\big)\\
    &\overset{\meqref{eq:aoco2}}{=}&\big(\mathrm{id}\otimes\mathcal{L}_{\cdot_{A}}(y)\big)(Q\otimes\mathrm{id})\Delta\big(P(x)\big)=(Q\otimes\mathrm{id})\big(\mathrm{id}\otimes\mathcal{L}_{\cdot_{A}}(y)\big)\Delta\big(P(x)\big).
\end{eqnarray*}
Thus \meqref{eq:bialg1} holds. Similarly, \meqref{eq:bialg2}-\meqref{eq:bialg7} hold. Therefore  $(A,\triangleright_{A},\triangleleft_{A},\vartheta,\theta)$ is a \sdppb.
\end{proof}

\begin{rmk}\label{rmk:bialgebra}
The fact that an averaging commutative and cocommutative
infinitesimal bialgebra induces a \sdppb can
also be observed from Proposition \mref{pro:5.2} by the
equivalence between Manin triples and bialgebras. Moreover, such
observation is also helpful in proving that a \sdppb
gives rise to an anti-pre-Lie bialgebra, since there is a
one-to-one correspondence between anti-pre-Lie bialgebras and
Manin triples of Lie algebras associated to commutative
$2$-cocycles (\cite[Corollary 2.16]{LB2024}).
\end{rmk}

This implication and the related structures can be summarized in the following diagram.

\vspace{-.4cm}
{\tiny
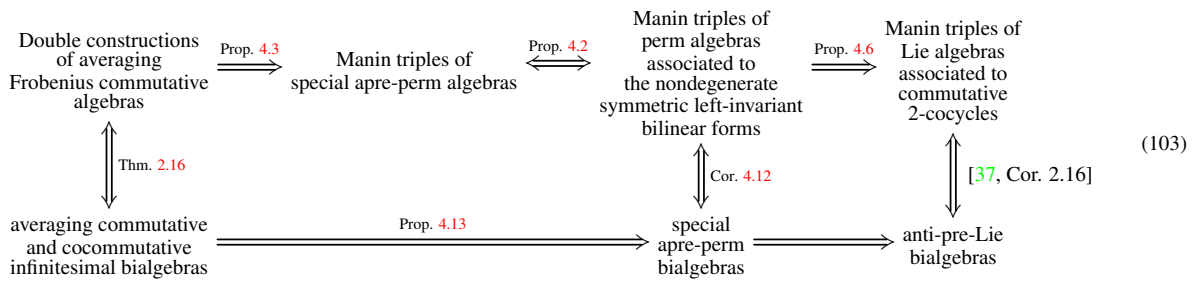
    \begin{equation} \label{eq:rbdiag}
\begin{split}
      \xymatrix{
         \txt{Double constructions \\of averaging \\Frobenius commutative\\ algebras}
            \ar@{=>}[r]^-{{\rm Prop.}~\ref{pro:5.2}}
            \ar@{<=>}[d]^-{{\rm Thm.}~\ref{thm:2.11}}
              & \txt{Manin triples of\\ \sdpps}
              \ar@<.4ex>@{<=>}[r]^-{{\rm Prop.}~\ref{pro:equ}}
              & \txt{Manin triples of \\ perm algebras \\ associated to \\ the nondegenerate \\ symmetric left-invariant \\ bilinear forms}
              \ar@{=>}[r]^-{{\rm Prop.}~\ref{pro:antiManin}}
              \ar@{<=>}[d]^-{{\rm Cor.}~\ref{cor:5.8}}
              & \txt{Manin triples of \\ Lie algebras\\
             associated to \\ commutative \\ $2$-cocycles} \ar@{<=>}[d]^-{\text{\tiny ~\cite[Cor. 2.16]{LB2024}}}
             \\
            \txt{averaging commutative\\ and cocommutative \\ infinitesimal bialgebras}
            \ar@{=>}[rr]^-{{\rm Prop.}~\ref{pro:5.9} }
            & & \txt{special \\ apre-perm \\ bialgebras}
            \ar@{=>}[r]
            & \txt{anti-pre-Lie\\ bialgebras}}
\end{split}
\end{equation}
}

\noindent{\bf Acknowledgments.} This work is supported by NSFC
(12271265, 12261131498,  W2412041, 12401031), the Postdoctoral Fellowship
Program of CPSF under Grant Number GZC20240755, 2024T005TJ and
2024M761507,
the Fundamental Research Funds for the Central Universities and
Nankai Zhide Foundation.

\noindent
{\bf Declaration of interests. } The authors have no conflict of interest to declare that are relevant to this article.

\noindent
{\bf Data availability. } Data sharing is not applicable to this article as no data were created or analyzed.

\end{document}